\documentclass{amsart}

\usepackage{latexsym, xypic, hyperref, graphicx, color}
\usepackage{amsfonts, amssymb, amscd}
\xyoption{curve}

\newtheorem{theorem}{\bf Theorem}

\newtheorem{conjecture}{\bf Conjecture}

\newcommand{\R}{\mathbb{R}}
\newcommand{\Q}{\mathbb{Q}}
\newcommand{\Z}{\mathbb{Z}}
\newcommand{\C}{\mathbb{C}}
\newcommand{\IP}{\mathbb{P}}
\newcommand{\IH}{\mathbb{H}}

\newcommand{\rs}{\hat{\mathbb{\C}}}
\newcommand{\Teich}{\mbox{\rm Teich}}
\newcommand{\Mod}{\mbox{Mod}}
\newcommand{\PMod}{\mbox{PMod}}
\newcommand{\PGamma}{\mbox{P}\Gamma}
\newcommand{\Moduli}{\mbox{Moduli}}
\newcommand{\PSL}{\mbox{PSL}}
\newcommand{\Qbar}{\overline{\mathbb{Q}}}
\newcommand{\two}{\overset{2}{\to}}

\newcommand{\cW}{\mathcal{W}}
\newcommand{\HHH}{\mathcal{H}}
\newcommand{\AAA}{\mathcal{A}}
\newcommand{\UUU}{\mathcal{U}}
\newcommand{\HS}{\mathcal{HS}}

\newcommand{\co}{\colon}

\newcommand{\zm}{\mu}
\newcommand{\zs}{\sigma}
\newcommand{\zd}{\delta}
\newcommand{\zg}{\gamma}
\newcommand{\zG}{\Gamma}
\newcommand{\zv}{\varphi}
\newcommand{\zf}{\phi}
\newcommand{\zF}{\Phi}
\newcommand{\zb}{\beta}
\newcommand{\zl}{\lambda}
\newcommand{\zL}{\Lambda}

\newcommand{\zJ}{\Psi}
\newcommand{\za}{\alpha}

\newcommand{\zr}{\rho}
\newcommand{\zh}{\eta}
\newcommand{\zt}{\tau}

\newcommand{\mtwo}[4]                            
{\mbox{$\left[\begin{array}{cc}                  
#1 & #2 \\
#3 & #4
\end{array}
\right]$}}

\title{Origami, affine maps, and complex dynamics}

\subjclass[2010]{Primary: 37F10;  Secondary: 57M12, 30F60}

\keywords{Thurston map, branched covering, Teichm\"{u}ller theory,
self-similar group}

\begin{author}[Floyd]{William Floyd}
\email{floyd@math.vt.edu}
\address{Department of Mathematics, Virginia Tech, Blackskburg, VA 24061 USA}
\end{author}

\begin{author}[Kelsey]{Gregory Kelsey}
\email{gkelsey@bellarmine.edu}
\address{
Department of Mathematics, Bellarmine University, 2001 Newburg Rd., Louisville, KY 40205}
\end{author}

\begin{author}[Koch]{Sarah Koch}
\email{kochsc@umich.edu}
\address{Department of Mathematics, University of Michigan, Ann Arbor, MI 48109 USA}
\end{author}

\begin{author}[Lodge]{Russell Lodge}
\email{russell.lodge@stonybrook.edu}
\address{Institute for Mathematical Sciences, 100 Nicolls Rd., Stony Brook, NY 11794-3660 USA}
\end{author}

\begin{author}[Parry]{Walter Parry}
\email{walter.parry@emich.edu}
\end{author}

\begin{author}[Pilgrim]{Kevin M. Pilgrim}
\email{pilgrim@indiana.edu}
\address{Department of Mathematics, Indiana University, Bloomington, IN 47405 USA}
\end{author}

\begin{author}[Saenz]{Edgar Saenz}
\email{easaenzm@vt.edu}
\address{Department of Mathematics, Virginia Tech, Blackskburg, VA 24061 USA}
\end{author}

\begin{document}

\maketitle

\begin{center}
{\em Version \today}
\end{center}

\begin{abstract} We investigate the combinatorial and dynamical
properties of so-called {\em nearly Euclidean Thurston maps}, or {\em NET maps}.  These 
maps are perturbations of many-to-one folding maps of an affine
two-sphere to itself.  The close relationship between NET maps and
affine maps makes computation of many invariants tractable.  In
addition to this, NET maps are quite diverse, exhibiting
many different behaviors.  We discuss data, findings, and new
phenomena. 
\end{abstract}

\newpage
\tableofcontents

\newpage
 
\section{Introduction}
\label{secn:introduction}

Complex dynamics studies iteration of rational functions $f: \rs \to
\rs$.  An important subclass consists of the {\em postcritically
finite} rational maps: those for which the {\em postcritical set}
$P(f):=\cup_{n>0}f^{\circ n}(C(f))$ is finite; here
$C(f)$ is the finite set of points at which $f$ is not
locally injective. For example, if $c_R \in \C$ is the unique root of
$c^3+2c^2+c+1$ with $\Im(c_R)>0$, then the quadratic polynomial
$f(z)=z^2+c_R$, known as {\em Douady's rabbit}, is postcritically
finite: it has one fixed critical point at infinity, and the unique
finite critical point at the origin is periodic of period 3. Another
example is provided by $f(z)=z^2+i$.  The Julia sets of these maps are
shown in Figure~\ref{fig:julias}.  Their three finite postcritical
points are marked with tiny circles.

\begin{figure}
\begin{center}\includegraphics[width=2in]{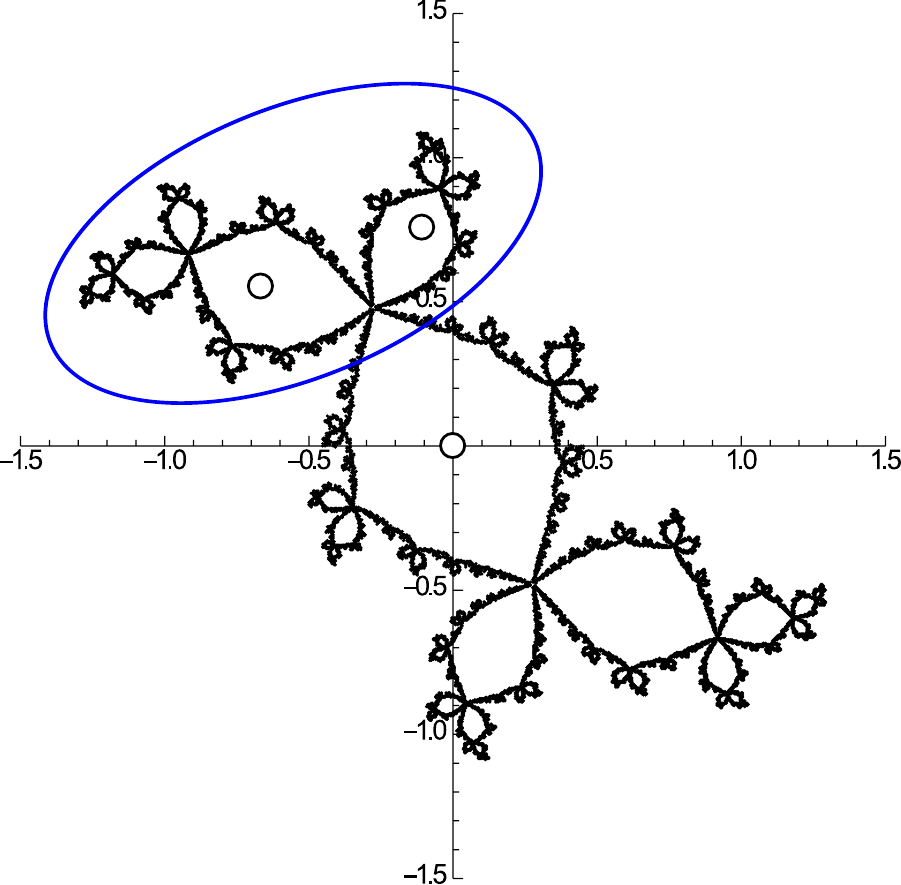} 
\includegraphics[width=2in]{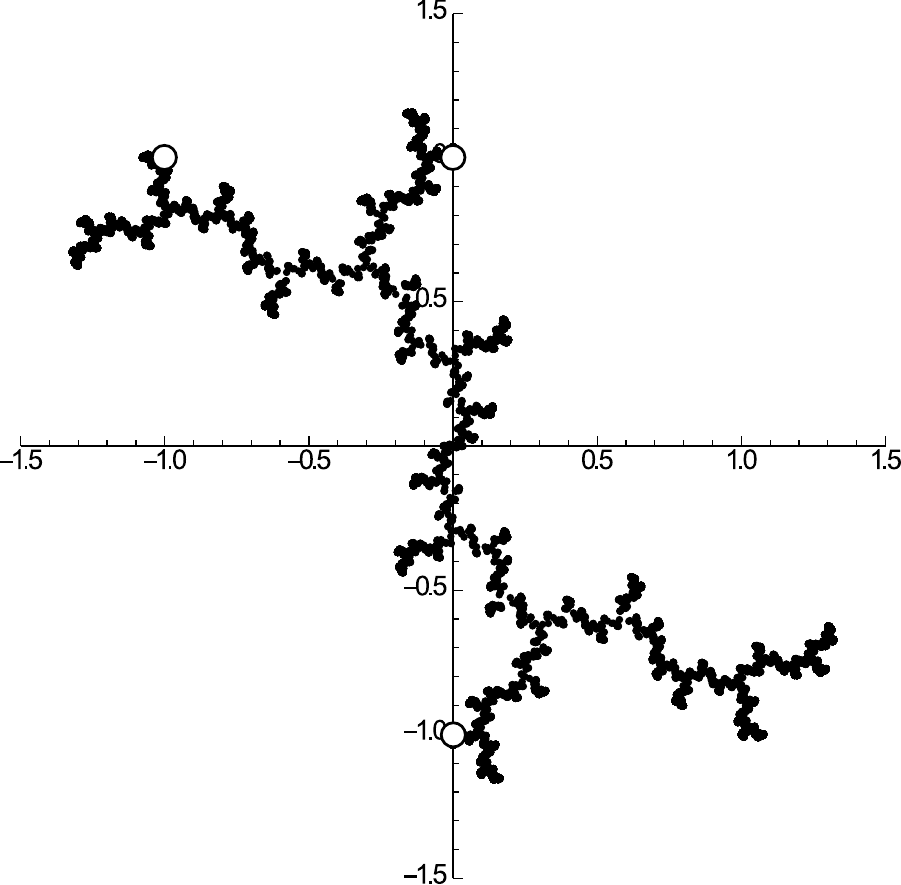}\end{center}
\caption{Julia sets of the rabbit $z\mapsto z^2+c_R$ and $z\mapsto
z^2+i$.}\label{fig:julias}
\end{figure}

\subsection*{Thurston Maps} A {\em Thurston map} is a continuous,
orientation-preserving branched covering $f:S^2\to S^2$ of degree at least two for which the
set $P(f)$ is finite. For example, if $h$ is a Dehn twist about the
blue ellipse in Figure~\ref{fig:julias}, one may {\em twist} Douady's
rabbit by post-composing $f$ with $h$ to yield a Thurston map $g=h\circ f$.  More generally, if $h_0,
h_1\co S^2\to S^2$ are orientation-preserving homeomorphisms such that
$h_0$ agrees with $h_1^{-1}$ on $P(f)$
, then we call $h_0 \circ f \circ h_1$ a {\em twist} of $f$.
We call the resulting collection of maps the {\em pure Hurwitz class}
of $f$.  See \S \ref{sec:hurwitz} for related definitions and
discussion.

\subsection*{Combinatorial equivalence} Two Thurston maps $f, g$ are
{\em combinatorially equivalent} or \emph{Thurston equivalent} if
there are orientation-preserving homeomorphisms $h_0, h_1:
(S^2, P(f)) \to (S^2, P(g))$ for which $h_0 \circ f = g \circ
h_1$ and $h_0, h_1$ are isotopic through homeomorphisms
agreeing on $P(f)$.  More succinctly: they are conjugate up to isotopy
relative to their postcritical sets.  This is related to a more
familiar notion. For a finite set $P \subset S^2$, denote by
$\PMod(S^2, P)$ the pure mapping class group of the pair $(S^2,P)$.
Suppose $P(f)=P(g)=P$.  The notion of combinatorial equivalence
between $f$ and $g$ is analogous to the notion of conjugacy in
$\PMod(S^2, P)$, but now the representing maps are branched coverings
instead of homeomorphisms.

\subsection*{W. Thurston's characterization of rational maps}
W. Thurston \cite{DH} gave necessary and sufficient combinatorial
conditions for a Thurston map $f$ to be equivalent to a rational map
$g$.  The statement has two cases, depending on the Euler characteristic $\chi(\mathcal{O}(f))$ of a certain orbifold structure $\mathcal{O}(f)$ on the sphere associated to the dynamics of $f$ on the set $P(f) \cup C(f)$; see \cite{DH}. Typical Thurston maps have hyperbolic orbifold ($\chi<0$) and checking rationality involves ruling out certain families of curves, called \emph{obstructions}.  Atypical Thurston maps have Euclidean orbifold ($\chi=0$)--we call these \emph{Euclidean}--and checking rationality involves examining the eigenvalues of a two by two matrix.  Apart from a well-understood subset of Euclidean maps known as {\em flexible rational Latt\`es maps}, the rational map $g$ equivalent to a Thurston map $f$, if it exists, is unique up to
holomorphic conjugacy.  

Checking that there are no obstructions is often very
difficult.  To give a sense of the complexity that can occur, consider
the following result, Theorem \ref{thm:rabbitdendrite}.  All of the
maps involved are typical NET maps---the special class of Thurston maps that
is the focus of this paper.

\begin{theorem} \label{thm:rabbitdendrite} Each twist of the rabbit
$f(z)=z^2+c_R$ is combinatorially equivalent to a complex polynomial
$z^2+c$ where $c^3+2c^2+c+1=0$; all three cases arise.  In contrast,
for twists of $f(z)=z^2+i$, the problem of determining rationality of
a twist $h\circ f$ reduces to checking the image of $h$ under a
homomorphism to a finite group of order $100$. Among these
combinatorial classes there are precisely two classes of rational
maps, namely $z \mapsto z^2\pm i$, and a countably infinite family of
pairwise inequivalent twists of the form $h^n\circ g, n \in \Z$, where
$g$ is a particular obstructed twist and $h$ is a Dehn twist about the
obstruction of $g$.
\end{theorem}

The first
statement follows from a general result now known as the {\em
Berstein-Levy theorem} \cite{lev}, while the second is more recent and
is one of the main results of Bartholdi and Nekrashevych's article
\cite[\S6]{BN}.

When all twists of $f$ are equivalent to rational maps, we say its
pure Hurwitz class is {\em completely unobstructed}.  Some pure
Hurwitz classes are completely obstructed (defined analogously) and
some are neither, i.e. contain both obstructed and unobstructed maps.

\subsection*{Induced dynamics on curves} A simple closed curve in
$S^2-P(f)$ is {\em essential} if it is not freely homotopic to a constant curve at a point in $S^2-P(f)$.
An essential curve is {\em peripheral} if it is essential but homotopic into
arbitrarily small neighborhoods of a point of $P(f)$. A Thurston map
$f$ and all its iterates are unramified outside the set $P(f)$, so
curves in $S^2-P(f)$ can be iteratively lifted under $f$.  For example, it is easy
to see that under iterated pullback the blue ellipse $\gamma$ in
Figure~\ref{fig:julias} is periodic of period 3 up to homotopy, and
that $\deg(f^3: \tilde{\gamma} \to \gamma)=4$, where $\tilde{\gamma}$
is the unique preimage of $\gamma$ under $f^{3}$ that is essential and
nonperipheral in $S^2-P(f)$.  There are countably infinitely many
simple closed curves up to homotopy in $S^2-P(f)$, though, and it is
 harder to see the following:

\begin{theorem}[{\cite[Theorem 1.6]{kmp:tw}}]\label{thm:rabbit} 
Under iterated pullback of the rabbit polynomial $f(z)=z^2+c_R$, any
simple closed curve becomes either inessential or peripheral in
$S^2-P(f)$ or, up to homotopy, falls into the above 3-cycle.
\end{theorem}

See the end of \S \ref{sec:fga} for an outline of another way to
prove this result.

Focusing on the behavior of curves under pullback is important.  The
statement of W. Thurston's characterization theorem for rational maps
among Thurston maps says that obstructions to $f$ being rational are
multicurves $\Gamma \subset S^2-P(f)$ with a certain invariance
property. Specifically: after deleting inessential and peripheral
preimages, we have $f^{-1}(\Gamma)\subset \Gamma$ up to homotopy in
$S^2-P(f)$, and the spectral radius of a certain associated linear map
$f_\Gamma: \R^\Gamma \to \R^\Gamma$ is greater than or equal to $1$.

\subsection*{Teichm\"uller theory} The proof of W. Thurston's
characterization theorem reduces the question ``Is $f$ equivalent to a
rational map?'' to the problem of finding a fixed point for a certain
holomorphic self-map $\sigma_f: \Teich(S^2, P(f)) \to \Teich(S^2,
P(f))$ of a Teichm\"uller space, given by pulling back complex
structures under $f$; see Douady and Hubbard \cite{DH}.  For the
precise definition of $\sigma_f$, we refer the reader to \cite{BEKP}.
Although $\sigma_f$ is complicated and transcendental, it covers a
finite algebraic correspondence on moduli space: \[\xymatrix{ &
\Teich(S^2,P(f)) \ar[dd]_{\pi}\ar[rr]^{\sigma_f} \ar[dr]^{\omega} &
&\Teich(S^2, P(f)) \ar[dd]^{\pi} \\ &&\mathcal{W}\ar[dl]_Y \ar[dr]^X
&\\ & \Moduli(S^2, P(f))& & \Moduli(S^2, P(f))}\] See \cite{K} and
\cite{kps}.  In the above diagram, $Y$ is a finite covering, $X$ is
holomorphic, and only $\sigma_f$ depends on $f$; up to isomorphism induced by conjugation by impure mapping class elements,
the remainder depends
only on the pure Hurwitz class of $f$, cf. \cite{K} and \cite[\S
3]{kps}. In 
the case of the rabbit, the moduli space is isomorphic to
$\IP^1-\{0,1,\infty\}$, the map $X$ is injective, so that we may
regard $\mathcal{W} \subset \IP^1-\{0,1,\infty\}$, the map $Y$ is
given by $x \mapsto 1-\frac{1}{x^2}$, and $\mathcal{W}=\IP^1-\{\pm 1,
0,\infty\}$; see \cite{BN}.  For quadratics with four postcritical
points and hyperbolic orbifold, 
$X$ is always injective, and
the formulas for $Y \circ X^{-1}$ are quite simple. 
For other maps,  the
equations defining the correspondence $\mathcal{W}$ may be 
complicated.  This happens even for maps with four postcritical points, including Euclidean quadratics, many cubics, and most NET maps.

Of special interest is the group $G_f < \PMod(S^2,P(f))$ represented by {\em liftable} homeomorphisms $h$, i.e. those
for which there is a lift $\tilde{h}$ representing an element in
$\PMod(S^2, P(f))$ with $h\circ f = f \circ \tilde{h}$; the assignment
$h \mapsto \tilde{h}$ gives a homomorphism $\phi_f: G_f \to \PMod(S^2,
P(f))$ which we call the {\em virtual endomorphism} on $\PMod(S^2,
P(f))$.  If there is a fixed point $\tau$ of $\sigma_f$, and if
$w:=\omega(\tau), m:=Y(w)=X(w)$, then $\phi_f = X_*\circ Y^{-1}_*$ is
the induced map on fundamental groups based at these points.
The domain of $\zf_f$ is the subgroup $Y_*(\pi_1(\mathcal{W}),w)$.

\subsection*{Nearly Euclidean Thurston maps} The family of {\em nearly
Euclidean Thurston} (NET) maps, introduced in \cite{cfpp}, provides an
extremely rich family of simple examples of Thurston maps for which
explicit algorithmic computations are possible.  By definition, a
Thurston map $f$ is NET if (i) each critical point has local degree 2,
and (ii) $\#P(f)=4$.  So, both $z \mapsto z^2+c_R$ and $z \mapsto
z^2+i$ are NET maps.  A NET map is Euclidean if and only if  $P(f) \cap C(f)=\emptyset$.  One thing that makes NET maps so interesting is
that each NET map $f$ admits what we call a {\em NET map
presentation}. This means that $f$ is combinatorially equivalent to a
map in a very special normal form.  See \S \ref{sec:prens}.
Conversely, a NET map presentation defines a combinatorial
equivalence class of NET maps.

\subsection*{NET map presentation for the rabbit} Figure
\ref{fig:rabbitpren} shows a NET map presentation diagram for the
rabbit $f(z)=z^2+c_R$.  With some conventions understood, it is
remarkably simple. Here are the details.

  \begin{figure}
\centerline{\includegraphics{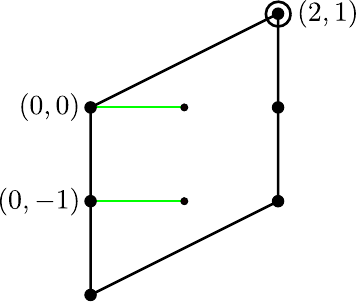}}
\caption{A NET map presentation diagram for the rabbit.}
\label{fig:rabbitpren}
  \end{figure}

Let $\Lambda_2=\mathbb{Z}^2$, and let $\Lambda_1<\Lambda_2$ be the
lattice generated by $(0,-1)$ and $(2,1)$.  For $i=1,2$ let $\Gamma_i$
be the groups generated by 180 degree rotations about the elements of
$\Lambda_i$, and let $S^2_i = \R^2/\Gamma_i$ be the quotient.  A
fundamental domain for $\Gamma_1$ is shown in
Figure~\ref{fig:rabbitpren}.  Since $\Gamma_1 < \Gamma_2$ there is an
``origami'' quotient map $\overline{id}: S^2_1 \to S^2_2$.  Let
$A=\left[\begin{smallmatrix}0 & 2 \\ -1 & 1 \end{smallmatrix}\right]$,
$b=\left[\begin{smallmatrix}2 \\ 1\end{smallmatrix}\right]$ and $\Phi:
\R^2 \to \R^2$ be the affine map $\zF(x)=Ax+b$.  The columns of $A$
are our lattice generators, and $b$ is the circled lattice point in
Figure~\ref{fig:rabbitpren}.  The map $\zF$ induces an affine
homeomorphism $\overline{\Phi}:S^2_2 \to S^2_1$.  We set
$g=\overline{\Phi}\circ\overline{id}$; it is an affine branched cover
of $S^2_1$ to itself. Finally, we put $f=h\circ g$ where $h: S^2_1 \to
S^2_1$ is a point-pushing homeomorphism along the indicated green
segments in Figure~\ref{fig:rabbitpren}.  Figure~\ref{fig:rabbitpren}
completely describes this Thurston map up to combinatorial
equivalence.

\subsection*{Computations for NET maps}
If $f$ is a NET map, then $\#P(f)=4$. This makes things much easier than for
general Thurston maps.  After some natural identifications, we have the following. 
\begin{enumerate}
  \item The Teichm\"uller space $\Teich(S^2, P(f))$ is the upper
half-plane $\IH \subset \C$.
  \item The pure and ordinary mapping class groups $\PMod(S^2, P(f))$
and $\Mod(S^2, P(f))$ are the congruence subgroup $\PGamma(2)$ and
$\PSL(2,\Z)\ltimes (\Z/2\Z\times\Z/2\Z)$ respectively.
  \item The domain of the correspondence $\mathcal{W}$ is a classical
modular curve (see \S \ref{subsecn:classical_modular}).  
  \item The map $Y:\mathcal{W}\to \mathrm{Moduli}(S^2,P(f))$ extends
to a Belyi map $\overline{Y}:{\overline{\mathcal{W}}}\to
\mathbb{P}^1$.
  \item The homotopy classes of curves in $S^2-P(f)$ are classified by
their slopes, that is, elements of $\Qbar = \Q \cup \{\pm
1/0=\infty\}$; with conventional identifications, each slope $p/q$
corresponds to the ideal boundary point $-q/p \in \partial \IH$.
  \item By taking preimages of curves, we obtain a slope function
$\mu_f: \Qbar \to \Qbar \cup \{\odot\}$ where $\odot$ denotes the
union of inessential and peripheral homotopy classes; this encodes the
Weil-Petersson boundary values of $\sigma_f$, shown to exist in general by Selinger
\cite{S}. More precisely, if $\frac{p}{q}\in \overline{\Q}$ and
$\zs_f(-\frac{q}{p})\in \overline{\Q}$, then
$\zm_f(\frac{p}{q})=-\zs_f(-\frac{q}{p})^{-1}$.  If
$\zs_f(-\frac{q}{p})\notin \overline{\Q}$, then
$\zm_f(\frac{p}{q})=\odot$.
\item Varying the choice of translation term $b$ does not affect the
fundamental invariants, such as $\zs_f$ above, or the ones given below
in Theorem \ref{thm:slopefn} (so long as such choices result in maps with four postcritical points, which is almost always the case).  Thus {\em virtual NET map presentation diagrams}, in which the translation term is omitted, suffice to compute such invariants.

\end{enumerate}
Since NET maps are very close to affine maps, it turns out that
explicit computations of what happens to slopes under pullback are
possible.  

\begin{theorem}[{\cite[Theorems 4.1, 5.1, 5.3]{cfpp}}]
 \label{thm:slopefn} Given a NET map presentation for a Thurston map
$f$ and the slope $p/q$ of a curve $\gamma$, there is an
algorithm which computes
\begin{enumerate}
\item $c_f(p/q)=$ the number of essential and nonperipheral preimages $\tilde{\gamma}_1, \ldots, \tilde{\gamma}_c \subset f^{-1}(\gamma)$,
\item $d_f(p/q)=$ the common degree by which these preimages map onto $\gamma$, and
\item $\mu_f(p/q)=$ the slope of the common homotopy class of the preimages $\tilde{\gamma}_i$.
\end{enumerate}
\end{theorem}

The behavior of the slope function $\mu_f$ is rather intricate.  For
the rabbit, Figure~\ref{fig:rabbitslopefn} is a plot of the values
$\zm_f(\frac{p}{q})$ with $\left|p\right|\le 50$ and $0\le q\le 50$.
In fact, the closure of this graph is all of $\R^2$.  

\begin{figure}
\begin{center}\includegraphics[scale=.4]{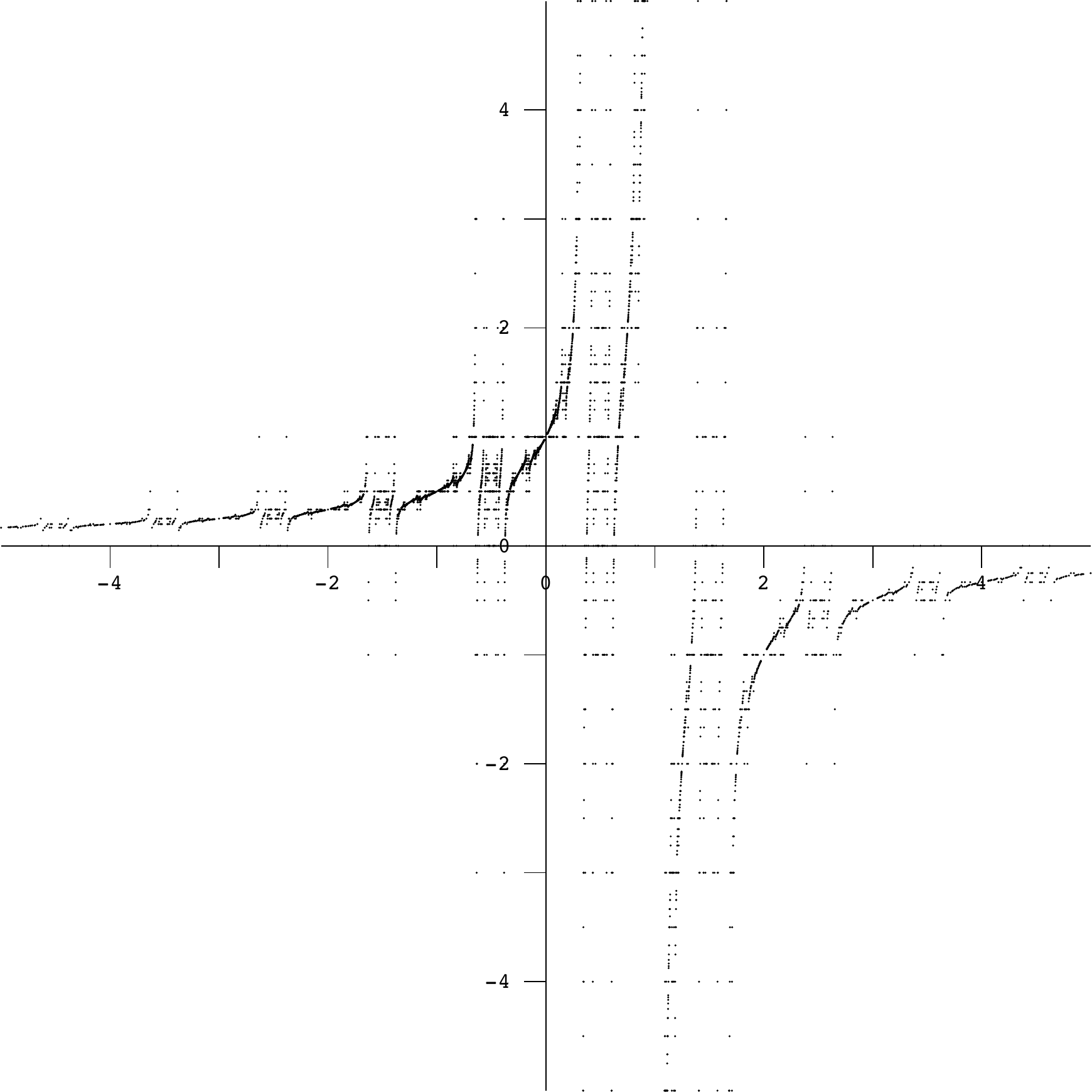}\end{center}
\caption{A portion of the graph of the rabbit's slope function.}
\label{fig:rabbitslopefn}
\end{figure}

The {\em multiplier} of $p/q$ under $f$ is
$\delta_f(p/q)=c_f(p/q)/d_f(p/q)$.  When $\#P(f)=4$, W. Thurston's
characterization theorem reduces to the following.  A Thurston map $f$
with hyperbolic orbifold is obstructed if and only if there exists a
slope $p/q$ for which $\mu_f(p/q)=p/q$ and $\delta_f(p/q) \geq 1$.  It
turns out that knowledge of data points of the form $(p/q, p'/q',
\delta_f(p/q))$, where $p'/q'=\mu_f(p/q)$, restricts the possible
slopes of such obstructions:
\begin{theorem}[Half-Space Theorem {\cite[Theorem 6.7]{cfpp}}]
\label{thm:halfspace} Suppose $p'/q'=\mu_f(p/q) \neq p/q$ or
$\odot$. There is an algorithm that takes as input the triple $(p/q,
p'/q', \delta_f(p/q))$ and computes as output an {\em excluded open
interval} $J=J(p/q, p'/q', \delta_f(p/q)) \subset \Qbar$ containing
$-q/p$ such that no point of $J$ is the negative reciprocal of the
slope of an obstruction.
\end{theorem}
The intervals in Theorem~\ref{thm:halfspace} are obtained from
half-spaces in $\IH$.  The boundary of a half-space in $\IH$ has a
finite part, consisting of points in $\IH$, and an infinite part,
consisting of points in $\partial \IH$.  The intervals in
Theorem~\ref{thm:halfspace} are the infinite boundary points of
half-spaces in $\IH$ minus endpoints. Figure~\ref{fig:rabbithalfsp}
shows the deployment of some of the half-spaces (in grey, with black
boundaries, $\left|p\right|\le 25$, $\left|q\right|\le 25$) for these
excluded intervals in the case of the presentation of the rabbit in
Figure~\ref{fig:rabbitpren}.

\begin{figure}
\begin{center}\includegraphics{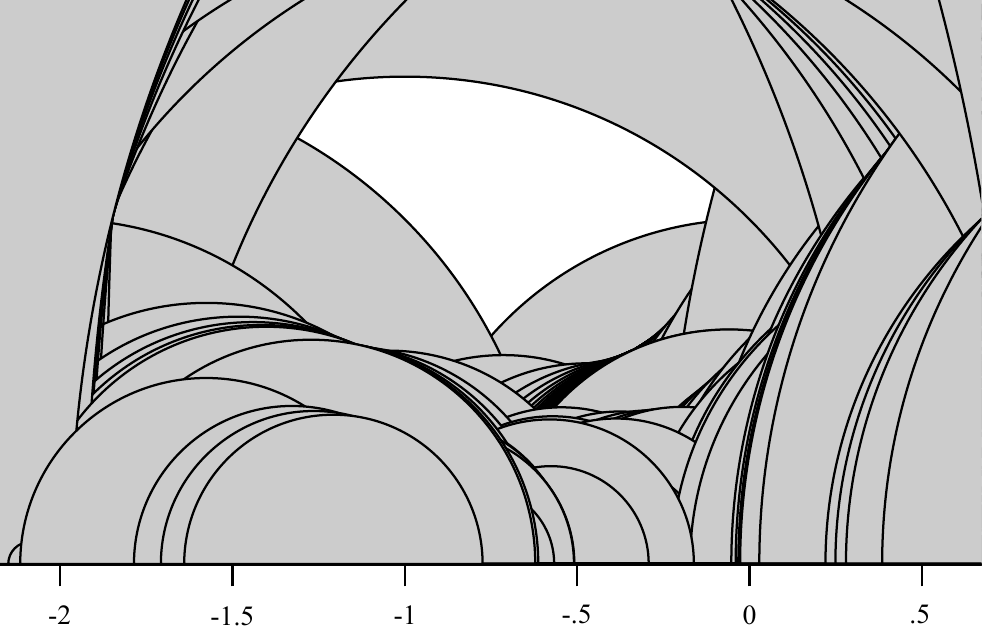} \end{center}
\caption{Half-spaces for the rabbit.}\label{fig:rabbithalfsp}
\end{figure}

It may happen that for some finite set $\{p_1/q_1, \ldots, p_m/q_m\}$
the associated excluded intervals cover all of $\Qbar$. This implies
that there are no obstructions and therefore, by W. Thurston's
characterization theorem, $f$ is equivalent to a rational map.  As
indicated by Figure~\ref{fig:rabbithalfsp}, finitely many excluded
intervals cover $\overline{\mathbb{Q}}$ for the rabbit.  In fact,
careful inspection shows that three half-spaces suffice.  An extension
of this theorem to the case when $\mu_f(p/q)=p/q$ or $\odot$ is
described in \S \ref{sec:extdhalfsp}.  See the discussion of
\S \ref{sec:extdhalfsp} in the introduction below.  

Parry with assistance from Floyd has written and continues to improve
a computer program {\tt NETmap} which computes  information like the above 
for a given NET map.  Figures~\ref{fig:rabbitslopefn}
 and
\ref{fig:rabbithalfsp} are part of this program's output for the
rabbit with the presentation in Figure~\ref{fig:rabbitpren}.  
That
it can do what it does illustrates the tractability of NET maps.
Executable files, documentation and more can be found at \cite{NET}.

\subsection*{Summary} Here is a summary of this article.

\subsubsection*{Findings (\S \ref{sec:findings})} We briefly report on
the phenomena observed among the NET maps we have investigated.

\subsubsection*{NET map presentations (\S \ref{sec:prens})} This
section explains presentations of general NET maps of the type given
above for the rabbit.

\subsubsection*{Hurwitz classes (\S \ref{sec:hurwitz})} We first
briefly recall some terminology and facts related to Hurwitz classes.
We present invariants of Hurwitz classes of NET maps, in particular, a
complete set of invariants for impure Hurwitz classes of NET maps. 
For NET maps, an impure Hurwitz class consists either entirely of Euclidean maps, or of non-Euclidean maps. 

Parry has written a computer program which enumerates these impure
Hurwitz class invariants and outputs one representative virtual NET
map presentation (See Section~\ref{sec:prens} for presentations.) For
every impure Hurwitz class of NET maps.  It organizes these virtual
NET map presentations by elementary divisors.  (See the discussion of
Hurwitz invariants in \S~\ref{sec:hurwitz} for the definition of their elementary divisors.)
The NET map web site \cite{NET} contains a catalog of these
representative NET map presentations through degree 30.  It also
contains {\tt NETmap}'s 
output for every
such example.  We use the notation $mn$HClass$k$ to denote the $k$th
virtual NET map presentation with elementary divisors $m$ and $n$ in
this catalog.

We prove the following theorem.

\begin{theorem}\label{thm:twists} Suppose $f$ is a non-Euclidean NET map and $\HHH$
its impure Hurwitz class.  There is an algorithm which computes the
image of $\delta_f$. This image $\delta(\HHH)$ depends only on $\HHH$
and not on the choice of representative $f$.  Furthermore: 
\begin{enumerate}
\item $\delta(\HHH)=\{0\} \iff \sigma_f$ is constant; 
\item $\delta(\HHH) \subset [0,1)\iff \HHH$ is
completely unobstructed; 
\item $\delta(\HHH)\ni 1 \iff \HHH$ contains infinitely many distinct 
combinatorial classes.
\end{enumerate}
\end{theorem}

There are analogous statements for general Thurston maps and pure
Hurwitz classes.

We discuss instances of statement 1 in Finding \ref{constsigma} of
Section~\ref{sec:findings}.  We discuss instances of statement 2 in
Findings \ref{cmpyuobstrd1} and \ref{cmpyuobstrd2}.  We discuss
instances of when $\HHH$ contains only finitely many distinct
combinatorial classes in Finding \ref{finnumeqclass}.

In \S \ref{subsecn:classical_modular}, we also relate the
correspondence $\cW$ to classical modular curves.

\subsubsection*{Invariants of degree 2 NET maps (\S
\ref{sec:enumeration})} We discuss invariants of degree 2 NET maps.
The complete classification for quadratics has recently been completed
by Kelsey and Lodge \cite{KL}.  In \cite{fpp2}, a classification of
dynamical portraits for NET maps is given; these classify the
corresponding pure Hurwitz classes in degrees 2 and 3.

\subsubsection*{A conformal description of $\zs_f$ for a degree 2
example (\S \ref{sec:sigmaf})} This section demonstrates the
tractability of NET maps.  It illustrates how numerous invariants of
NET maps can be computed by doing so for a specific example.  We show
for this example that the pullback map is the analytic continuation,
via repeated reflection, of a conformal map between ideal hyperbolic
polygons.

For this, recall that the slope function $\mu_f$ encodes the boundary
values of $\sigma_f$, and that lifting under $f$ determines a virtual
endomorphism $\phi_f: \PMod(S^2, P(f)) \dashrightarrow \PMod(S^2,
P(f))$ with domain $G_f$. By enlarging $G_f$ to include reflections
and so extending $\phi_f$, and noting that reflections must lift to
reflections, we can sometimes obtain detailed information about both
$\phi_f$ and $\sigma_f$.  In this way exact calculations of certain
values of $\sigma_f$ are sometimes possible.  Along similar lines, a
perhaps remarkable feature is that if $\mu_{f}(p/q)=\odot$, in some
circumstances, exploiting the structure of functional equations
involving reflections yields exact calculations of the limiting
behavior of $\sigma_{f}(\tau)$ as $\tau \to -q/p$ from within a
fundamental domain of $G_{f}$.  This analysis is done for a
particular example in \S \ref{sec:sigmaf}.  It is done for the rabbit
near the end of \S \ref{sec:fga}.

More generally, the pullback map of every quadratic NET map is the
analytic continuation, via repeated reflection, of a conformal map
between ideal hyperbolic polygons.  This is also surely true of all
pullback maps of cubic NET maps.  In all of these cases {\tt NETmap}
reports that the subgroup of liftables in the extended modular group
(which allows reversal of orientation) acts on $\IH$ as a reflection
group.  However, there are NET maps with degree 4 (41HClass3) for
which the extended modular group liftables does not act on $\IH$ as a
reflection group.  This seems to be the predominant behavior in higher
degrees.  In such cases we do not understand the behavior of $\zs_f$
as well.

\subsubsection*{Dynamics on curves in degree 2 (\S \ref{sec:fga})}
Relying heavily on the results of \S \ref{sec:sigmaf}, we investigate
the dynamics on the set of homotopy classes of curves under iterated
pullback of quadratic NET maps with one critical postcritical point.
For a map $\mu: X \to X$ from a set $X$ to itself, we say a subset $A$
of $X$ is a {\em finite global attractor} if $A$ consists of finitely
many cycles into which each element $x \in X$ eventually iterates. We
show that for maps in this class that are rational, there is a finite
global attractor containing at most four slopes, while for obstructed
maps, there may be either (a) a finite global attractor; (b) an
infinite set of fixed slopes with no wandering slopes; or (c) a finite
set of fixed slopes coexisting with wandering slopes.

We remark that using techniques from self-similar groups, Kelsey and
Lodge \cite{KL} have accomplished this for all quadratic rational maps
$f$ with $\#P(f)=4$ and hyperbolic orbifold.

\subsubsection*{The extended half-space theorem (\S
\ref{sec:extdhalfsp})} The half-space theorem,
Theorem~\ref{thm:halfspace}, applies to all extended rational numbers $r$
which are mapped to different extended rational numbers by the
Thurston pullback map $\zs_f$ of a NET map $f$.  The half-space
theorem provides an explicit interval about $r$, called an excluded interval,
which contains no negative reciprocals of slopes of obstructions for
$f$.  If finitely many such excluded intervals cover a cofinite subset
of $\partial \mathbb{H}$, then we have only finitely many remaining
slopes to check to determine whether $f$ is combinatorially equivalent
to a rational map.  

Computations using {\tt NETmap} suggest
that there exist many NET maps for which every finite union of
excluded intervals omits an interval of real numbers.  Under suitable hypotheses, this can be proved rigorously.
For example, consider the NET map $f$ with the presentation diagram in
Figure~\ref{fig:omit}.  It is rational since the algorithms show
$\delta_f(p/q) \in [0,1)$ for all $p,q$. However, one can prove that
every excluded interval arising from the half-space theorem, Theorem~\ref{thm:halfspace}, is bounded. This implies that every finite union of such intervals fails to cover all of the boundary.

  \begin{figure}
\centerline{\includegraphics{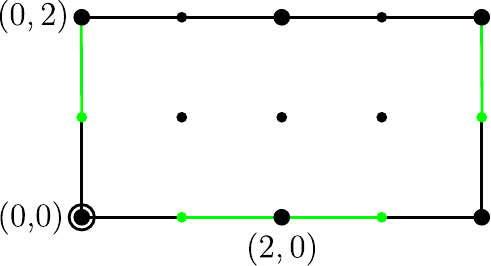}}
\caption{A presentation diagram for a NET map whose rationality cannot
be detected by the half-space theorem but whose rationality is detected by the extended half-space theorem.}
\label{fig:omit}
  \end{figure}

Saenz \cite{SM} establishes rationality of his main example by finding 
infinitely many excluded intervals.
This provides motivation to extend the half-space theorem to extended
rational numbers $r$ such that either $\zs_f(r)=r$ or $\zs_f(r)\in
\mathbb{H}$.  This is what the extended half-space theorem does.  It
does not actually provide any new excluded intervals.  Instead it
provides a way to construct an explicit union of infinitely many excluded
intervals such that this union is a deleted neighborhood of a given
extended rational number $r$ such that either $\zs_f(r)=r$ or
$\zs_f(r)\in \mathbb{H}$.

One feature of the computer program {\tt NETmap} is to implement a
straightforward algorithm based on the extended half-space theorem.
In practice, it almost always determines whether or not a NET map is,
or is not, equivalent to a rational map.  This leads us to

\begin{conjecture}\label{ehsadecides} Suppose $f$ is a NET map. Then
the extended half-space algorithm decides, in finite time, whether or
not $f$ is equivalent to a rational function.
\end{conjecture}

For general Thurston maps, that such an algorithm exists in theory is announced in \cite{BBY}.  That such an algorithm exists in practice is announced in \cite[Algorithm V.8]{BD} and indeed this is what Bartholdi's program \cite{B} attempts to do. 

It is clear that every extended rational number is either in an
excluded interval or in a deleted excluded interval.  So if every
irrational number is contained in an excluded interval, then
compactness of $\partial \IH$ implies Conjecture~\ref{ehsadecides}.

We outline a proof of the extended half-space theorem in \S
\ref{sec:extdhalfsp}.

\section{Findings} 
\label{sec:findings}

We report here on many findings of interest for NET maps.
\begin{enumerate}
\item\label{fga1} It is conjectured (see e.g. \cite[\S 9]{L}) that 
for non-Latt\`es rational maps, the pullback relation on curves has a 
finite global attractor. For NET maps, our evidence suggests that, more 
generally, this holds if there do not exist obstructions with multiplier 
equal to 1.  The converse, however, is false.
More precisely, in \S \ref{sec:fga} we prove Theorem~\ref{thm:fga},
which shows that a NET map $f_0$ introduced in \S \ref{sec:sigmaf}
with virtual presentation 21HClass1 is obstructed and its pullback
map on curves has a finite global attractor which consists of just the
obstruction.  Theorem~\ref{thm:fga} also shows that there exist many
obstructed maps without finite global attractors.

\item\label{fixedpts} By perturbing flexible Latt\`es maps slightly
within the family of NET maps, we can build examples of NET maps whose
slope functions have many fixed points.  By perturbing other
Latt\`{e}s maps, we can build examples of NET maps whose slope
functions have cycles of lengths 2, 3, 4 or 6; other examples yield
5--11 and 13--15, inclusive.  We do not know whether all cycle lengths
occur.

\item\label{notriv} There are many examples of NET maps with
hyperbolic orbifolds for which no curve has all of its preimages
trivial: Example 3.1 of \cite{cfpp}; the main example of \cite{L}; all
NET maps in impure Hurwitz classes represented by the following
virtual presentations: 22HClass6; 31HClass 5, 6, 9; 51HClass 14--16,
23, 25.  This property is equivalent to surjectivity of $X$
\cite[Theorem 4.1]{kps}. Indeed, among such examples there occur those
whose pure (even impure) Hurwitz class is completely unobstructed
(22HClass6), completely obstructed (31HClass9 with translation term
$\zl_1$), and mixed-case obstructed (Example 3.1 of \cite{cfpp} and
the main example of \cite{L}).

\item\label{finnumeqclass} There exist NET maps $f$ whose impure
Hurwitz class $\HHH$ contains only finitely many Thurston equivalence
classes, some of which are obstructed and some of which are not.
Statement 3 of Theorem~\ref{thm:twists} shows that this is equivalent
to the existence in $\zd(\HHH)$ of some multipliers which are less than
1, some multipliers which are greater than 1 but none equal to 1.
This occurs for 41HClass6, 8, 11, 19, 24.

\item\label{matings} The operation of {\em mating} takes the dynamics of two polynomials and glues them together 
to form a Thurston map; see e.g. \cite{Mi}.  Given a Thurston map, it might be expressible as a mating in multiple ways. The NET map of Figure \ref{fig:dompren} arises as a mating in at least $\zf(2m+1)$ ways, where $\zf$ is Euler's totient function; that of Figure \ref{fig:pren29} in at least $m+1$ ways. This is established by showing that $\mu_f$ has at least the corresponding number of {\em equators}--fixed-points of maximal multiplier, with the additional condition of preserving orientation--and appealing to \cite[Theorem 4.2]{Mey}.

  \begin{figure}
\centerline{\includegraphics{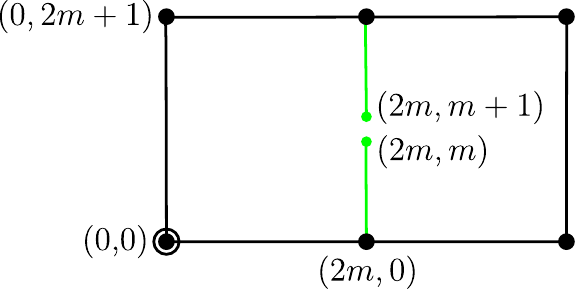}}
\caption{A presentation diagram for an NET map that arises as a mating in at least $\phi(2m+1)$ ways, where $\phi$ is Euler's totient function.}
\label{fig:dompren}
  \end{figure}

  \begin{figure}
\centerline{\includegraphics{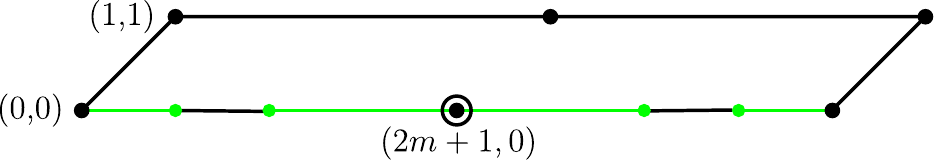}}
\caption{A presentation diagram for a NET map with
degree $2m+1$ and dynamic portrait 29 which arises as a mating in at least $m+1$ ways.}
\label{fig:pren29}
  \end{figure}

\item\label{eqsigmas} 
There are at least two simple ways to create pairs of combinatorially
inequivalent Thurston maps $f$, $g$ for which $\zs_f=\zs_g$:
\begin{enumerate} \item Let $g$ be any NET map, let $h$ be a flexible
Latt\`{e}s map with $P(h)=P(g)$ and let $f=g\circ h$.  Since $\zs_h$
is the identity map, $\zs_f=\zs_g$.  \item The translation term in the
affine map $\zF$ in the definition of NET map presentation does not
affect $\zs_f$.  For example, it turns out that changing the
translation term in a NET map presentation for $f(z)=z^2+i$ obtains a
NET map $g$ whose dynamic portrait is different from that of $f$, but
$\zs_f=\zs_g$.  Thus typically, a given NET map $f$ has three 
cousins sharing the same induced map $\zs_f$.

\end{enumerate}
However, there exist other examples: the NET maps with virtual presentation 41HClass19 have the same
pullback maps as those with virtual presentation 41HClass24.  These
phenomena suggest that non-dynamical, Hurwitz-type invariants might be
viewed as more fundamental than dynamic portraits.

\item\label{constsigma} It is possible for a NET map $f$ to have a
constant pullback map $\zs_f$.  Examples are given in \cite[\S
10]{cfpp} and \cite[Chap. 5]{SM}.  The property that $\zs_f$ is
constant depends only on the impure Hurwitz class of $f$ and hence
only on its Hurwitz structure set (Section~\ref{sec:hurwitz}, Hurwitz
invariants).

Proposition 5.1 of \cite{BEKP} provides a way to construct Thurston
maps $f$ with constant pullback maps.  Very briefly, the idea here is
that if $f=g\circ s$ and the pullback map of $s$ maps to a trivial
Teichm\"{u}ller space, then the pullback map of $f$ is constant.  We
refer to the hypotheses of this Proposition 5.1 as McMullen's
condition.

To describe the NET maps which satisfy McMullen's condition, we define
two types of Hurwitz structure sets.  Let $G$ be a finite Abelian
group generated by two elements such that $G/2G\cong (\Z/2\Z)\oplus
(\Z/2\Z)$.  Let $\HS$ be a Hurwitz structure set in $G$.  We say that
$\HS$ is an MC2 Hurwitz structure set if $\HS=\{\pm a\}\amalg \{\pm
b\}\amalg \{\pm c\}\amalg \{\pm d\}$, where both $a$ and $b$ have
order 4 and $2a=2b=c-d$.  We say that $\HS$ is an MC4 Hurwitz
structure set if $\HS=\{\pm a\}\amalg \{\pm b\}\amalg \{\pm c\}\amalg
\{\pm d\}$, where both $a$ and $b$ have order 4, $2a\ne 2b$, $c=a+b$
and $d=a-b$.

The following theorem essentially answers the question of what NET
maps satisfy McMullen's condition.  Its proof will appear elsewhere.

\begin{theorem}\label{thm:mcmullen} A NET map is impurely Hurwitz
equivalent to a NET map which satisfies McMullen's condition if and
only if its Hurwitz structure set is either an MC2 or MC4 Hurwitz
structure set.
\end{theorem}

Now that we essentially know what NET maps satisfy McMullen's
conditions, what NET maps have constant pullback maps?  We do not know
the answer, but we have the following.  We say that a NET map $f$ is
imprimitive if there exist NET maps $f_1$ and $f_2$ such that $f_1$ is
Euclidean, $f=f_1\circ f_2$ and the postcritical sets of $f$, $f_1$
and $f_2$ are equal.  In this case the pullback map of $f$ is constant
if and only if the pullback map of $f_2$ is constant.  We say that $f$
is primitive if it is not imprimitive.  These notions extend to
Hurwitz structure sets.  We have found five equivalence classes of
primitive Hurwitz structure sets whose NET maps have constant pullback
maps but which do not satisfy McMullen's condition.  Here are
representatives for them.
  \begin{equation*}
\{\pm (1,0),\pm (1,1),\pm (7,1),\pm (3,2)\}\subseteq
(\Z/8\Z)\oplus (\Z/4\Z)\quad \deg(f)=8
  \end{equation*}
  \begin{equation*}
\{\pm (2,0),\pm (0,2),\pm (2,2),\pm (4,2)\}\subseteq
(\Z/6\Z)\oplus (\Z/6\Z)\quad \deg(f)=9
  \end{equation*}
  \begin{equation*}
\{\pm (1,0),\pm (0,1),\pm (5,1),\pm (2,2)\}\subseteq
(\Z/6\Z)\oplus (\Z/6\Z)\quad \deg(f)=9
  \end{equation*}
  \begin{equation*}
\{\pm (1,0),\pm (0,1),\pm (1,2),\pm (4,1)\}\subseteq
(\Z/6\Z)\oplus (\Z/6\Z)\quad \deg(f)=9
  \end{equation*}
  \begin{equation*}
\{\pm (1,0),\pm (1,2),\pm (11,2),\pm (3,3)\}\subseteq
(\Z/12\Z)\oplus (\Z/6\Z)\quad \deg(f)=18
  \end{equation*}

The second of these is the degree 9 example in \cite[\S 10]{cfpp} and
\cite[Chap. 5]{SM}.  This leads us to make the following conjecture.

\begin{conjecture}\label{cnje:neversepg} The Hurwitz structure set of
every primitive NET map with constant pullback map is either an MC2 or
MC4 Hurwitz structure set or it is equivalent to one of the above five
exceptional Hurwitz structure sets.
\end{conjecture}

This conjecture has been verified by computer for all NET maps with
first elementary divisor $m\le 300$. In particular, it has been
verifed for all NET maps  with degree at most
300.

\item\label{gn} Table~\ref{tab:genus} gives the possibilities for the
genus and number of cusps of $\mathcal{W}$ for all NET maps with
degree at most 8.  See \S \ref{subsecn:classical_modular} for a
discussion of the relationship between $\cW$, classical modular
curves, and Teichm\"uller curves.  Using the Riemann-Hurwitz formula,
one can show that \[ \deg(Y)=2(g-1)+n.\] Note the entries $(0,3)$ for
which $Y: \mathcal{W} \to \Moduli(S^2, P(f))$ is an isomorphism.  The one
in degree 4 arises from flexible Latt\`es maps.  The one in degree 8
arises from compositions $g_2 \circ g_1$, where $g_1$ is a quadratic
Thurston map with three postcritical points and $g_2$ is a flexible
Latt\`es map; note that every mapping class element lifts under $g_2$.

\begin{table}
\begin{center}\begin{tabular}{c|ccccccc}
$d$      &   2    &   3    &   4      &    5     &    6     &    7     &8 \\ \hline
$(g,n)$ & (0,4) & (0,6) & (0,3)   & (1,6)   & (0,6)   & (2,6)   & (0,3)\\
        &         &         & (0,4)   & (1,12) & (0,10) & (4,18) & (0,4)\\
        &         &         & (0,6)   &           & (1,8)   &           & (0,6)\\
        &         &         & (0,10) &           & (1,16) &           & (0,10)\\
        &         &         &           &           &           &      & (1,8)\\
        &         &         &           &           &           &      & (1,16)\\
        &         &         &           &           &           &      & (2,14)\\
        &         &         &           &           &           &      & (5,24)\\
\end{tabular}\end{center}
\medskip
\caption{A table of all possible ordered pairs $(g,n)$, where $g$ is
the genus and $n$ is the number of cusps of $\mathcal{W}$ with
degree $d\le 8$.}\label{tab:genus} \end{table}

\item\label{cmpyuobstrd1} There are many examples of NET maps $f$ for
which $\sigma_f$ is nonconstant and for which the impure Hurwitz class
is completely unobstructed: 22HClass1, 4--6; 31HClass7.  Indeed, the
impure Hurwitz class of almost every NET map which is a push of a
flexible Latt\`{e}s map is completely unobstructed, and the associated
pullback maps are nonconstant.  The NET map with presentation diagram
in Figure~\ref{fig:omit} is an example of this.  However, all NET maps
whose pullback maps are nonconstant and for which the impure Hurwitz
class is completely unobstructed seem to have degrees of the form
$n^2$, $2n^2$, $3n^2$ or $6n^2$.  We have verified this by computer
through degree 100.

\item\label{cmpyuobstrd2} Among quadratic pure Hurwitz classes, we
observe that being completely unobstructed is equivalent to the
condition that the inverse $Y \circ X^{-1}$ of the correspondence
extends to a postcritically finite hyperbolic rational map $g_f: \IP^1
\to \IP^1$ whose postcritical set consists of the three points at
infinity in $\Moduli(S^2, P(f)$.

\end{enumerate}

\section{NET map presentations}\label{sec:prens}

We next describe NET map presentations.  This section expands on the
discussion of a NET map presentation for the rabbit in the
introduction.  Details can be found in \cite{fpp1}. We begin with the
lattice $\zL_2=\mathbb{Z}^2$, a proper sublattice $\zL_1$ and an
orientation-preserving affine isomorphism $\zF\co \mathbb{R}^2\to
\mathbb{R}^2$ such that $\zF(\zL_2)=\zL_1$.  Let $\zG_1$ be the group
of isometries of $\mathbb{R}^2$ of the form $x\mapsto 2\zl\pm x$ for
some $\zl\in \zL_1$.  This information determines a Euclidean Thurston
map $g\co \mathbb{R}^2/\zG_1\to \mathbb{R}^2/\zG_1$ as in the
introduction's discussion of a NET map presentation for the rabbit.
The postcritical set $P_1$ of $g$ is the image of $\zL_1$ in
$\mathbb{R}^2/\zG_1$.  The image of $\zL_2-\zL_1$ in
$\mathbb{R}^2/\zG_1$ is the set of critical points of $g$.  To
describe $g$, all we need is to express $\zF$ as $\zF(x)=Ax+b$, where
$A$ is a $2\times 2$ matrix of integers and $b$ is an integral linear
combination of the columns, $\zl_1$ and $\zl_2$, of $A$.  We may even
assume that $b$ is either 0, $\zl_1$, $\zl_2$ or $\zl_1+\zl_2$.  Then
$\zl_1$, $\zl_2$ and $b$ determine $g$ up to Thurston equivalence.

The parallelogram $F_1$ with corners 0, $2\zl_1$, $\zl_2$ and
$2\zl_1+\zl_2$ is a fundamental domain for the action of $\zG_1$ on
$\mathbb{R}^2$.  The points of $\zL_1$ in $F_1$ are 0, $\zl_1$,
$2\zl_1$, $\zl_2$, $\zl_1+\zl_2$ and $2\zl_1+\zl_2$.  These six points
map onto $P_1$.  We choose six line segments (possibly trivial, just a
point) whose union is the full inverse image in $F_1$ of four disjoint
arcs $\zb_1$, $\zb_2$, $\zb_3$, $\zb_4$ in $\R^2/\zG_1$.  Each of the
six line segments joins one of 0, $\zl_1$, $2\zl_1$, $\zl_2$,
$\zl_1+\zl_2$, $2\zl_1+\zl_2$ and an element of $\Z^2$.  We call them
the {\em green line segments}, and we call $\zb_1$, $\zb_2$, $\zb_3$,
$\zb_4$ the {\em green arcs}.

Recall that if $\zb$ is an oriented arc in a surface, then a {\em
point-pushing homeomorphism along $\zb$} is a homeomorphism which is
the terminal homeomorphism of an isotopy of the surface supported in a
regular neighborhood of $\zb$ that pushes the starting point of $\zb$
to its ending point along $\zb$.  We have four green arcs $\zb_1$,
$\zb_2$, $\zb_3$, $\zb_4$.  Pushing along each $\zb_i$ determines, up
to homotopy rel $P_1$, a ``push map'' homeomorphism $h\co
\mathbb{R}^2/\zG_1\to \mathbb{R}^2/\zG_1$ which pushes along $\zb_1$,
$\zb_2$, $\zb_3$, $\zb_4$ in $\mathbb{R}^2/\zG_1$ from $P_1$ to a set
$P_2$ of four points in the image of $\zL_2$.

Now that we have $g$ and $h$, we set $f=h\circ g$.  This is a Thurston
map.  It is a NET map if it has four postcritical points, in which
case its postcritical set is $P_2$.  This fails only in special
situations when the degree of $f$ is either 2 or 4. (See the second
paragraph of Section 2 of \cite{cfpp} for more on this.)
Every NET map can
be expressed as a composition of a Euclidean map and a push map in
this way.  We call this a NET map presentation of $f$.  The result of
omitting the translation term $b$ from a NET map presentation is by
definition a {\em virtual NET map presentation}. 
The program {\tt NETmap} takes as input
a virtual NET map presentation.

So every NET map can be described up to Thurston equivalence by a
simple diagram.  This diagram consists of first the parallelogram
$F_1$.  This determines $\zl_1$ and $\zl_2$ and therefore the matrix
$A$.  Second, one of the elements 0, $\zl_1$, $\zl_2$, $\zl_1+\zl_2$
in $F_1$ is circled to indicate the translation term $b$.  Third, the
(nontrivial) green line segments are drawn in $F_1$.  We call this a
NET map presentation diagram.  Figure~\ref{fig:rabbitpren} is such a
diagram for the rabbit.

Note that the group $\text{SL}(2,\Z)$ acts naturally on NET map
presentation diagrams: given $P \in \text{SL}(2,\Z)$, transform the
entire diagram by application of $P$.  In \cite{fpp1} it is shown that
this corresponds to postcomposition by the element of the modular
group determined by $P$.

\section{Hurwitz classes}\label{sec:hurwitz}

\subsection*{Hurwitz equivalence}  Let $f,f'\co S^2\to S^2$ be
Thurston maps with postcritical sets $P =P(f)$ and $P'=P(f')$.  We
say that $f$ and $f'$ belong to the same modular group Hurwitz class
if there exist orientation-preserving homeomorphisms $h_0,h_1\co
(S^2,P)\to (S^2,P')$ such that $h_0\circ f=f'\circ h_1$.  If in
addition $h_0$ and $h_1$ agree on $P$, then we say that $f$ and $f'$
belong to the same pure modular group Hurwitz class.  For brevity, we
usually speak of pure and impure Hurwitz classes.

\subsection*{Proof of Theorem \ref{thm:twists}} Statements 1 and 2 of
Theorem~\ref{thm:slopefn} imply that there is an algorithm which
computes $c_f(p/q)$ and $d_f(p/q)$ for every slope $p/q$.  Theorem 4.1
of \cite{cfpp} implies that these values depend only on the image of
the ordered pair $(q,p)$ in $\mathbb{Z}_{2m}\oplus \mathbb{Z}_{2n}$
once $\zL_2/2\zL_1$ is appropriately identified with
$\mathbb{Z}_{2m}\oplus \mathbb{Z}_{2n}$.  This proves that there is an
algorithm which computes the image of $\zd_f$.

Let $h\co (S^2,P(f))\to (S^2,P(f))$ be an orientation-preserving
homeomorphism.  Then $h$ induces by pullback a bijection $\zm_h$ on
slopes.  Let $s$ be a slope.  Then $c_{h\circ f}(s)=c_f(\zm_h(s))$ and
$d_{h\circ f}(s)=d_f(\zm_h(s))$.  Also, $c_{f\circ h}(s)=c_f(s)$ and
$d_{f\circ h}(s)=d_f(s)$.  This proves the second assertion of
Theorem~\ref{thm:twists}.

We now establish the three final assertions.  Statement 1 follows from
\cite[Theorem 5.1]{kps}.  Statement 2 follows from W. Thurston's
characterization theorem and the observation that the impure modular
group acts transitively on slopes.  

We now turn to the necessity in statement 3.  Let $\simeq$ denote the
equivalence relation on $\HHH$ determined by isotopy rel $P=P(f)$.  So
$\HHH/\!\!\simeq$ is the set of isotopy classes of maps in the impure
Hurwitz class of $f$; in what follows, we write equality for equality
in this set.

The full modular group $\Mod(S^2, P)$ acts on $\HHH/\!\!\simeq$ both
by pre-composition and post-composition.  The set of combinatorial
classes in $\HHH$ is in bijective correspondence with the orbits of
the induced conjugation action of $\Mod(S^2, P)$ on $\HHH/\!\!\simeq$.
Since the pure mapping class group $\PMod(S^2, P)$ has finite index in
$\Mod(S^2, P)$, it suffices to show that there are infinitely many
orbits under the conjugation action of $\PMod(S^2, P)$.

The assumption $1 \in \delta_f(\HHH)$ implies that there exist $h_0,
h_1$ representing elements of $\Mod(S^2, P)$ for which $f_*:=h_0fh_1$
has an obstruction given by a curve $\gamma$ with multiplier equal to
$1$. Let $T$ be a (full, not half) Dehn twist about $\gamma$.  By
\cite[Theorem 9.1]{kps} there is a smallest positive integer $k$ such
that $T^k$ commutes with $f_*$ up to isotopy relative to 
$T^kf_*=f_*T^k$ in $\HHH/\!\!\simeq$. Let $k'$ be the smallest
positive integer for which $T^{k'}$ lifts under $f_*$ to an element of $\Mod(S^2, P)$.  Since Dehn
twists must lift to Dehn twists, and $f_*$ leaves $\gamma$ invariant,
we have $T^{k'}f_*=f_*T^{k''}$ for some $k'' \in \mathbb{Z}$.  Since
$k'$ is minimal, $k=nk'$ for some positive integer $n$. Thus
$f_*T^{nk'}=f_*T^k=T^kf_*=T^{nk'}f_*=f_*T^{nk''}$.  The right action
of $\PMod(S^2, P)$ on $\HHH/\!\!\simeq$ is free (\cite[\S 3]{kmp:tw}
or \cite[Prop. 4.1]{kam}) and so $k=k'$.

For $n \in \mathbb{Z}$ let $f_n:=f_*T^n$. We claim that for $n \neq m
\in \mathbb{Z}$, the maps $f_n$ and $f_m$ are not conjugate via an
element of $\PMod(S^2, P)$.  We argue somewhat similarly as in
\cite[\S 9]{kps}.  Suppose as elements of $\HHH/\!\!\simeq$ we have
$hf_n=f_mh$ for some $h \in \PMod(S^2, P)$.  The (class of) curve
$\gamma$ is the unique obstruction for both $f_n$ and $f_m$, so $h$
must fix the class of $\gamma$. Since $h \in \PMod(S^2, P)$, $h$ is a power of $T$,
say $T^l$.  Then $hf_n=f_mh \implies T^lf_*T^n=f_*T^mT^l$.  This
equation implies that $T^l$ lifts under $f_*$ to a pure mapping class
element and so by the previous paragraph $l=qk$ for some
$q$. Continuing, we have $f_*T^mT^l =T^lf_*T^n=
T^{qk}f_*T^n=f_*T^nT^{qk}\implies f_*T^{qk}T^n=f_*T^mT^{qk}\implies
f_*T^{n+l}=f_*T^{m+l}\implies n=m$, again by freeness of the right
action.

To prove sufficiency, suppose $1 \not\in \delta({\mathcal{H}})$.  It
suffices to show there are only finitely many combinatorial classes of
obstructed maps.  We use the combination and decomposition theory
developed in \cite{kmp:combinations} and outline the main ideas.  

Suppose $f \in \mathcal{H}$ has
an obstruction, $\gamma$.  Let $A_0$ be an annulus which is a regular neighborhood of $\gamma$ and put $\AAA_0:=\{A_0\}$. By altering $f$ within its homotopy class we may assume that the collection $\AAA_1$ of preimages of $A_0$ containing essential nonperipheral preimages of $\gamma$ is a collection of essential subannuli of $A_0$ with $\partial \AAA_1 \supset \partial A_0$; the restriction $f: \AAA_1 \to \AAA_0$ is the set of {\em annulus maps} obtained by decomposing $f$ along $\{\gamma\}$. Similarly, if we set $\UUU_0$ to be the collection of two components of $S^2-\AAA_0$, we obtain a collection of {\em sphere maps} $f: \UUU_1 \to \UUU_0$. Capping the holes of the sphere maps to disks yields a collection of Thurston maps with three or fewer postcritical points, and these are rational, by Thurston's characterization. 

As $f$ varies within $\mathcal{H}$, the collection of sphere maps varies over a finite set of combinatorial equivalence classes, since they are rational. The hypothesis that the multiplier of the obstructions are not equal to one implies that the set of combinatorial equivalence classes of annulus maps vary over a finite set as well.   The additional set-theoretic gluing data needed to reconstruct $f$ from the sphere an annulus maps also varies over a finite set. By \cite[Theorem ???]{kmp:combinations}, given two maps $f_1, f_2$ presented as a gluing of sphere and annulus maps, a combinatorial equivalence between gluing data, a collection of annulus maps, and a collection of sphere maps yields a combinatorial equivalence between $f_1$ and $f_2$. We conclude that the set of possible equivalence classes of maps $f$ in $\mathcal{H}$ is finite.

\qed

\subsection*{Modular groups} From the discussion of the NET
 map presentation of the rabbit (\S 1) and of NET map
presentations (\S 3), recall the definition of the quotient sphere
$S^2_2:=\R^2/\Gamma_2$. Equipped with its Euclidean half-translation
structure, it is a ``square pillowcase'' with four corners given
by the images of the lattice $\Lambda_2=\Z^2$.  Let $P$ be this set of
corners.  In this case the modular group $\text{Mod}(S^2,P)$ and pure
modular group $\text{PMod}(S^2,P)$ have the forms \[ \Mod(S^2, P)
\cong \PSL(2,\Z)\ltimes (\Z/2\Z\times\Z/2\Z)\] and \[ \PMod(S^2, P)
\cong P\Gamma(2)=\{A \equiv I \bmod 2\} \subseteq \PSL(2,\Z).\] The
group $\Mod(S^2, P)$ is then the group of orientation-preserving
affine diffeomorphisms, and $P\Gamma(2)$ the subgroup fixing the
corners pointwise. The group $\PSL(2,\Z)$ fixes the corner
corresponding to the image of the origin, but permutes the other three
corners so as to induce the natural action of the symmetric group
$S_3$.

\subsection*{Hurwitz invariants} Useful invariants of a NET map are
its {\em elementary divisors}, which we now define. Suppose $f$ is a
NET map given by a NET map presentation with affine map $x \mapsto
Ax+b$, where $A$ is an integral matrix whose determinant equals the
degree of $f$.  There are $P, Q \in \text{SL}(2,\Z)$ and positive
integers $m$ and $n$ such that $n|m$ and
$PAQ=\left[\begin{smallmatrix}m & 0 \\ 0 & n
\end{smallmatrix}\right]$.  The integers $m,n$ are unique and are the
{\em elementary divisors} of $A$; the matrix on the right is the {\em Smith normal form}.  They form an impure Hurwitz
invariant.  In fact, according to \cite[Theorem 5.5]{fpp2}, NET maps
$f$ and $g$ have equal elementary divisors if and only if $g$ is
Thurston equivalent to a NET map of the form $\zv\circ f$ for some
homeomorphism $\zv\co S^2\to S^2$.  (Note that $\zv$ need not
stabilize $P(f)$.)

For NET maps, a complete invariant of impure Hurwitz classes can be
given in terms of the {\em Hurwitz structure set}, $\mathcal{HS}$. To
define this, we use the usual data and Euclidean groups associated to
a NET map $f$.  The torus $\mathbb{R}^2/2\zL_1$ is a double cover of
the sphere $\mathbb{R}^2/\zG_1$.  The pullback of $P(f)$ in
$\mathbb{R}^2/2\zL_1$ is a subset of the finite group
$\mathbb{Z}^2/2\zL_1$, and it is a disjoint union of the form
$\HS=\{\pm h_1\}\amalg\{\pm h_2\}\amalg\{\pm h_3\}\amalg\{\pm h_4\}$.
This is a Hurwitz structure set.  More generally, let $G$ be a finite
Abelian group such that $G/2G\cong (\Z/2\Z)\oplus (\Z/2\Z)$.  A
Hurwitz structure set in $G$ is a disjoint union of four sets of the
form $\{\pm h\}$, where $h\in G$.  Returning to $\HS$, if $\zL'_1$ is
a sublattice of $\mathbb{Z}^2$, then we say that $\HS$ is equivalent
to a Hurwitz structure set $\HS'$ in $\mathbb{Z}^2/2\zL'_1$ if and
only if there exists an orientation-preserving affine isomorphism
$\zJ\co \mathbb{Z}^2\to \mathbb{Z}^2$ such that $\zJ(\zL_1)=\zL'_1$
and the map which $\zJ$ induces from $\mathbb{Z}^2/2\zL_1$ to
$\mathbb{Z}^2/2\zL'_1$ takes $\HS$ to $\HS'$.  Theorem 5.1 of
\cite{fpp2} states that the equivalence class of $\HS$ under this
equivalence relation is a complete invariant of the impure Hurwitz
class of $f$.

\subsection*{Relating $\cW$ to classical modular curves}
\label{subsecn:classical_modular} Let $f$ be a NET map with
postcritical set $P(f)$. 
Recall the correspondence $X, Y: \cW \to
\Moduli(S^2, P(f))$ from \S \ref{secn:introduction}.  This correspondence
is essentially an impure Hurwitz invariant; see \cite[\S 2]{K}.  In
this section we explicitly relate the space $\cW$ to classical modular
curves.

Corollary 5.3 of \cite{fpp2} states that every impure Hurwitz class of
NET maps is represented by a NET map whose presentation matrix is
diagonal.  So to understand $\cW$, we may assume that the presentation
matrix of $f$ has the form $A=\left[\begin{smallmatrix} m& 0 \\ 0 &
n\end{smallmatrix}\right]$, where $m$ and $n$ are positive integers
with $n|m$ and $mn=\deg(f)$.  By definition, $m$ and $n$ are the
elementary divisors of $f$.  So the presentation of $f$ has lattices
$\zL_2=\mathbb{Z}^2$ and $\zL_1=\left<(m,0),(0,n)\right>$.  It also
has a Hurwitz structure set $\mathcal{HS}\subseteq\zL_2/2\zL_1$.  The
discussion at the end of \S 2 of \cite{fpp2} shows that the group
$G_f$ of pure liftables is isomorphic to the image in
$\text{PSL}(2,\mathbb{Z})$ of the group $\widehat{G}_f$ of all
elements $M$ in $\text{SL}(2,\mathbb{Z})$ such that $M \zL_1=\zL_1$,
$M\equiv 1\text{ mod } 2$ and the automorphism of $\Z^2/2\zL_1$
induced by $M$ fixes $\mathcal{HS}$ pointwise up to multiplication by
$\pm 1$.

We interrupt this discussion to define some subgroups of
$\text{SL}(2,\mathbb{Z})$.  Let $N$ be a positive integer.  The
principle congruence subgroup of $\text{SL}(2,\mathbb{Z})$ with level
$N$ is
  \begin{equation*} 
\zG(N)=\left\{\left[\begin{smallmatrix}a & b \\ c & d
\end{smallmatrix}\right]\in
\text{SL}(2,\mathbb{Z}):\left[\begin{smallmatrix}a & b \\ c & d
\end{smallmatrix}\right]\equiv \left[\begin{smallmatrix}1 & 0 \\ 0 & 1
\end{smallmatrix}\right]\text{ mod } N\right\}.
  \end{equation*}
A congruence subgroup of $\text{SL}(2,\mathbb{Z})$ is a subgroup which
contains $\zG(N)$ for some $N$.  Two such subgroups are
  \begin{equation*} 
\zG_0(N)=\left\{\left[\begin{smallmatrix}a & b \\ c & 
d\end{smallmatrix}\right]\in 
\text{SL}(2,\mathbb{Z}):c \equiv 0 \text{ mod } N\right\}
  \end{equation*}
and
  \begin{equation*} 
\zG^0(N)=\left\{\left[\begin{smallmatrix}a & b \\ c & 
d\end{smallmatrix}\right]\in 
\text{SL}(2,\mathbb{Z}):b \equiv 0 \text{ mod } N\right\}.
  \end{equation*}
Two others are
  \begin{equation*} 
\zG_1(N)=\left\{\left[\begin{smallmatrix}a & b \\ c & 
d\end{smallmatrix}\right]\in 
\zG_0(N):a \equiv d\equiv 1 \text{ mod } N\right\}
  \end{equation*}
and
  \begin{equation*} 
\zG^1(N)=\left\{\left[\begin{smallmatrix}a & b \\ c & 
d\end{smallmatrix}\right]\in 
\zG^0(N):a \equiv d\equiv 1 \text{ mod } N\right\}.
  \end{equation*}

  We next relate $\widehat{G}_f$ to these congruence subgroups.  The
condition that $M\equiv 1\text{ mod }2$ simply says that
$\widehat{G}_f\subseteq \zG(2)$.  We next interpret the condition
that $M \zL_1=\zL_1$.  Let $M=\left[\begin{smallmatrix} a& b \\ c & d
\end{smallmatrix}\right]$.  Since $(0,n)\in \zL_1$, we need that
$M\cdot (0,1)\in \zL_1$.  Equivalently, $(bn,dn)\in \zL_1$.  This
amounts to requiring that $bn\equiv 0\text{ mod } m$, that is,
$b\equiv 0\text{ mod }\frac{m}{n}$.  Since $n|m$, the condition that
$M\cdot (m,0)\in \zL_1$ is satisfied by every $M\in
\text{SL}(2,\mathbb{Z})$.  Hence the condition that $M \zL_1=\zL_1$ is
equivalent to the condition that $M\in \zG^0(\frac{m}{n})$.  Therefore
$\widehat{G}_f\subseteq \zG^0(\frac{m}{n})\cap \zG(2)$.  On the other
hand, the group $\zG^1(2m)\cap \zG_1(2n)$ stabilizes $\zL_1$ and acts
trivially on $\zL_2/2\zL_1$, and so it fixes $\mathcal{HS}$ pointwise.
Thus
  \begin{equation*}
\zG(2m)\subseteq \zG^1(2m)\cap \zG_1(2n)\subseteq
\widehat{G}_f\subseteq \zG^0(\tfrac{m}{n})\cap \zG(2).
  \end{equation*}

  In conclusion, let $\mathbb{H}^*$ be the Weil-Petersson completion
of $\mathbb{H}$.  Then $\cW$ is a modular curve such that
$\mathbb{H}^*/(\zG^1(2m)\cap \zG_1(2n))$ maps onto $\cW$ and $\cW$
maps onto $\mathbb{H}^*/(\zG^0(\frac{m}{n})\cap \zG(2))$.

\section{Invariants of degree 2 NET maps}\label{sec:enumeration}

We begin with a discussion of invariants of general Thurston maps and
then specialize to NET maps with degree 2.

The {\em dynamic portrait} is the directed graph with vertex set $C(f)
\cup P(f)$ and weighted edges \[x
\stackrel{\deg_x(f)}{\longrightarrow} f(x).\] The {\em static
portrait} is the bipartite directed graph whose vertex set is the
disjoint union of $A:=C(f) \cup P(f)$ and $B:=P(f)$, directed edges
\[x \stackrel{\deg_x(f)}{\longrightarrow}f(x)\] with $x\in A$, and the
elements of $A$ that lie in $P(f)$ are marked so as to distinguish
them from those elements of $A$ that do not lie in $P(f)$.  The {\em
augmented branch data} records, for each $y \in P(f)$, the partition
of $d=\deg(f)$ given by the collection of local degrees $\{\deg_x(f) :
f(x)=y\}$.  For example, the dynamic portrait of the rabbit is $(a
\two b \to c \to a, d\two d)$, the static portrait is $(p_1 \two q_1,
p_2 \two q_2, p_3 \to q_3, p_4\to q_4)$, and the branch data is
$([2],[2],[1,1],[1,1])$.  The static portrait, and hence
branch data, are impure Hurwitz invariants.  The dynamic portrait is a
pure Hurwitz invariant but not, in general, an impure Hurwitz
invariant.

For NET maps, dynamic and static portraits are completely classified
\cite{fpp2}.  Table~\ref{tab:dynportraits} gives the number of dynamic
portraits as a function of the degree.

\begin{table}
\begin{center}
\begin{tabular}{r|cccccccccc}
$d$      &   2    &   3    &   4      &    5     &    6     &    7     &8 & \\ \hline
$n$  & 16 & 94 & 272 & 144 & 338 & 152 & 476 &  \\  \hline\hline
$d \mod 4, d \geq 9$      &   0   &   1  &   2     &    3   \\ \hline
$n$  & 483 & 153 & 353 & 153 \\ 
\end{tabular}
\end{center}
\caption{The number $n$ of dynamic portraits among NET maps of degree $d$.}
\label{tab:dynportraits} \end{table}

\subsection*{Degree 2 NET maps}

Recall that a quadratic Thurston map is NET if and only if it has four
postcritical points.  In degree 2 there are 3 impure Hurwitz classes,
completely classified by static portrait or, equivalently, by the
number of critical points in the postcritical set; see
Theorem~\ref{thm:hclasses}.  Here are the three static portaits. We
label marked points in the domain by $p_i$, we label marked points in
the codomain by $q_i$, and we label unmarked critical points in the
domain by $c_i$.  \[
\begin{array}{l}
p_1\to q_1, p_2\to q_1, p_3\to q_2, p_4\to q_2, c_1\two q_3, c_2\two q_4\\
p_1 \to q_1, p_2 \to q_1, p_3 \two q_2, p_4 \to q_3, c_1 \two q_4\\
p_1 \two q_1, p_2 \two q_2, p_3\to q_3, p_4\to q_4
\end{array}
\]
Pure Hurwitz classes are completely classified by the corresponding dynamic
portraits. There are 16 of them. All but one is represented by rational
functions; the exception is $a \two b \to a$, $c\two d \to c$. 

Kelsey and Lodge \cite{KL} have completed the classification of
quadratic NET combinatorial classes.  They generalize the methods of
Bartholdi-Nekrashevych  on the twisted rabbit problem \cite{BN},
analyzing wreath recursions on the pure mapping class group.  These
wreath recursions are derived from the 16 correspondences on moduli
space.

\subsubsection*{Impure Hurwitz  classes in degree 2.}
In this section we prove Theorem~\ref{thm:hclasses}.  It shows
that two NET maps with degree 2 belong to the same impure 
Hurwitz class if and only if they have the same number of critical
postcritical points.  These maps all have two critical points and four
postcritical points.  So there might be either 0, 1 or 2 critical
postcritical points.  Hence there are three  impure Hurwitz 
classes of NET maps with degree 2.  The case in which there are no
critical postcritical points is exactly the case of the Euclidean NET
maps with degree 2.  They form one  impure Hurwitz  class.
Another is represented by $f(z)=z^2+i$ and the other is represented
by the rabbit, corabbit and airplane.  Here is the theorem.

\begin{theorem}\label{thm:hclasses} Two degree 2 NET maps belong to
the same  impure Hurwitz  class if and only if they have the same
number of critical postcritical points.
\end{theorem}
  \begin{proof} Let $f$ be a degree 2 NET map.  The elementary
divisors of $f$ are $m=2$ and $n=1$ because their product is
$\deg(f)=2$ and the second divides the first.  The group
$\Lambda_2/2\Lambda_1$ in the definition of Hurwitz structure set is
then isomorphic to $\Z_4 \oplus \Z_2$.  The Hurwitz structure set
$\mathcal{HS}$ of $f$ can thus be identified with a subset of
  \begin{equation*}
\mathbb{Z}_4\oplus \mathbb{Z}_2
=\{(0,0),\pm (1,0),(2,0),(0,1), \pm (1,1),(2,1)\}.
  \end{equation*}
Elements of order 1 or 2 in $\mathcal{HS}$ correspond to postcritical
points of $f$ which are not critical.  The other elements of
$\mathcal{HS}$ are paired by multiplication by $-1$, and these pairs
correspond to critical postcritical points.  It is shown in \cite{fpp2}
that two NET maps with equal Hurwitz structure sets
belong to the same  impure Hurwitz  class.  Moreover,
transforming $\mathcal{HS}$ either by an automorphism of
$\mathbb{Z}_4\oplus \mathbb{Z}_2$ which lifts to
$\text{SL}(2,\mathbb{Z})$ or by a translation by an element of order 2
preserves the  impure Hurwitz  class of $f$.

It is easy to see that if two NET maps belong to the same impure
Hurwitz class, then they have the same number of critical postcritical
points.  The converse statement is what must be proved.

Suppose that $f$ is a NET map with degree 2 and no critical
postcritical points.  Then the first paragraph of this proof shows
that there is only one possibility for the Hurwitz
structure set of $f$ after identifying it with a subset of
$\mathbb{Z}_4\oplus \mathbb{Z}_2$; it must be
$\{(0,0),(2,0),(0,1),(2,1)\}$.  So there is just one impure Hurwitz
class of such maps.

Next suppose that $f$ is a NET map with degree 2 and exactly one
critical postcritical point.  Then the Hurwitz structure set of $f$
contains three of the four elements of order 1 or 2 in
$\mathbb{Z}_4\oplus \mathbb{Z}_2$.  Since translation by an element of
order 2 preserves the  impure Hurwitz  class of $f$, we may
assume that $\mathcal{H}$ contains $(0,0)$, $(2,0)$, $(0,1)$ and
either $\pm (1,0)$ or $\pm (1,1)$.  Now we verify that
$\left[\begin{smallmatrix}1 & 0 \\ 1 & 1 \end{smallmatrix}\right]\in
\text{SL}(2,\mathbb{Z})$ induces an automorphism of
$\mathbb{Z}_4\oplus \mathbb{Z}_2$ which fixes $(0,0)$, $(2,0)$,
$(0,1)$ and interchanges $\pm (1,0)$ and $\pm (1,1)$.  Thus there is
only one equivalence class of these Hurwitz structure sets and only
one  impure Hurwitz  class of such maps.

Finally, suppose that $f$ is a NET map with degree 2 and two critical
postcritical points.  In this case $\mathcal{HS}$ must contain $\pm
(1,0)$ and $\pm (1,1)$ in addition to two of the four elements of
order 1 or 2.  Since translating $\mathcal{HS}$ by an element of order
2 preserves the  impure Hurwitz  class of $f$, we may assume that
$\mathcal{HS}$ contains $\pm (1,0)$, $\pm (1,1)$ and $(0,0)$.

Suppose in addition that $(2,0)\in \mathcal{HS}$.  Then
$\mathcal{HS}=\{(0,0),\pm (1,0),(2,0),\pm 1,1)\}$.  Example 10.3 of
\cite{cfpp} shows that $\mathcal{HS}$ is never separating
(nonseparating in the language there).  Between Lemma 10.1 and Theorem
10.2 of \cite{cfpp} it is shown that this implies that the Thurston
pullback map of $f$ is constant.  But Theorem 10.10 of \cite{cfpp}
shows that there does not exist a NET map with degree 2 whose Thurston
pullback map is constant.  (The Thurston maps for this choice of
$\mathcal{HS}$ have fewer than four postcritical points.)  Thus
$(2,0)\notin \mathcal{HS}$.

So $\mathcal{HS}$ contains either $(0,1)$ or $(2,1)$.  One verifies
that $\left[\begin{smallmatrix}1 & 2 \\ 0 & 1
\end{smallmatrix}\right]\in \text{SL}(2,\mathbb{Z})$ induces an
automorphism on $\mathbb{Z}_4\oplus \mathbb{Z}_2$ which fixes $(0,0)$,
$\pm (1,0)$, $\pm (1,1)$ and interchanges $(0,1)$ and $(2,1)$.  Hence
there is only one equivalence class of these Hurwitz structure sets
and only one  impure Hurwitz  class of such maps.

This proves Theorem~\ref{thm:hclasses}.
\end{proof}

\section{A conformal description of $\zs_f$ for a degree 2 example}
\label{sec:sigmaf}

In this section we discuss how in many, but not all, cases it is
possible in a sense to determine the pullback map $\zs_f\co
\mathbb{H}\to \mathbb{H}$ of a NET map $f$.  This description is like
that of the classical triangle functions, as discussed in Chapter 1 of
\cite{leh}.  For the triangle functions, we begin with a conformal
equivalence between a hyperbolic triangle and the upper half plane.
Here we begin with a conformal equivalence between two hyperbolic
triangles or, more generally, two hyperbolic polygons.  The map is
then extended to the entire hyperbolic plane using the reflection
principle.  

We will focus on a particular example map $f_0$.  Our discussion
involves several features of $f_0$.  For each feature, we provide
first a brief general discussion for arbitary NET maps $f$, and then
illustrate it using our example map $f_0$.
\medskip

\noindent\textbf{Example: the NET map $f_0$.} Our example map $f_0$ is
determined up to Thurston equivalence by the presentation diagram in
Figure~\ref{fig:sigmapren}.  We have lattices $\zL_2=\mathbb{Z}^2$,
$\zL_1=\left<(2,0),(0,1)\right>$ and Hurwitz structure set
  \begin{equation*}
\mathcal{HS}=\{(0,0),\pm (1,0),(2,0),(0,1)\}\subseteq \mathbb{Z}_4\oplus
\mathbb{Z}_2\cong \zL_2/2\zL_1.
  \end{equation*}
We also have a Euclidean NET map $g$ and a push map $h$ such that
$f_0=h\circ g$.
\medskip

  \begin{figure}
\centerline{\includegraphics{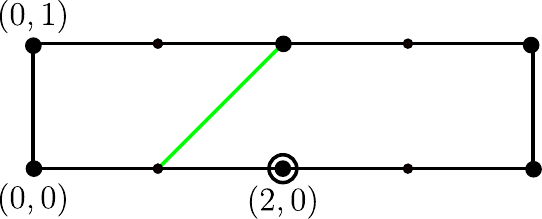}}
\caption{A presentation diagram for a NET map $f_0$; a vertical curve is an obstruction to rationality}
\label{fig:sigmapren}
  \end{figure}

\noindent\textbf{The subgroup $G_f$ of liftables in the extended
modular group $G$.}  Suppose $f$ now is an arbitrary NET map.  We work
with the extended modular group $G=\text{EMod}(S^2,P(f))$, which is
defined in the same way as the modular group $\text{Mod}(S^2,P(f))$
except that it allows reversal of orientation.  \S 2 of \cite{fpp2}
shows that $G$ is isomorphic to the group $\text{Aff}(2,\mathbb{Z})$
of all affine isomorphisms $\zJ\co \mathbb{R}^2\to \mathbb{R}^2$ such
that $\zJ(\mathbb{Z}^2)=\mathbb{Z}^2$ modulo the subgroup $\zG_2$ of
all maps of the form $x\mapsto 2\zl\pm x$ for some $\zl\in
\zL_2=\mathbb{Z}^2$.
\medskip

A map $\zv\co (S^2,P(f))\to (S^2,P(f))$ representing a homotopy class
in $G$ is liftable if there exists another such map $\widetilde{\zv}$
such that $\zv\circ f$ is homotopic to $f\circ \widetilde{\zv}$ rel
$P(f)$.  The subgroup of liftables for $f$ is the subgroup $G_f$ of
$G$ represented by all such liftable maps $\zv$.  \S 2 of
\cite{fpp2} shows that $G_f$ is isomorphic to the subgroup of $G$
whose elements lift to elements $\zJ\in \text{Aff}(2,\mathbb{Z})$ such
that $\zJ(\zL_1)=\zL_1$ and the map induced by $\zJ$ on $\zL_2/2\zL_1$
stabilizes $\mathcal{HS}$ setwise.  Turning to our example map $f_0$,
we let
  \begin{equation*}
\zJ_1(x)=\left[\begin{smallmatrix}1 & 0 \\ 0 & -1
  \end{smallmatrix}\right]x,
\quad
\zJ_2(x)=\left[\begin{smallmatrix}-1 & 0 \\ 2 & 1
  \end{smallmatrix}\right]x, 
\quad
\zJ_3(x)=\left[\begin{smallmatrix}1 & 2 \\ 0 & -1 
\end{smallmatrix}\right]x.
  \end{equation*}
These are elements of $\text{Aff}(2,\mathbb{Z})$ which stabilize $\zL_1$ and
$\mathcal{HS}$.  So $\zJ_1$, $\zJ_2$ and $\zJ_3$ determine elements
$\zr_1$, $\zr_2$ and $\zr_3$ of $G_{f_0}$. 

In this paragraph we show that $\zr_1$, $\zr_2$ and $\zr_3$ generate
$G_{f_0}$.  Let $G_{f_0}^+$ denote the subgroup of
orientation-preserving elements of $G_{f_0}$.  The discussion in \S
\ref{subsecn:classical_modular} which relates $\cW$ to classical
modular curves shows that if $\zJ(x)=Ax+b\in \text{Aff}(2,\mathbb{Z})$
is the lift of an element of $G_{f_0}^+$, then $A\in \zG^0(2)$.
Moreover, one easily verifies that there is no such $\zJ$ with
$A=\left[\begin{smallmatrix}1 & 0 \\ 1 & 1
  \end{smallmatrix}\right]$, an element of $\zG^0(2)$.  So the set of
these matrices arising as lifts of elements of $G_{f_0}^+$ is a proper subgroup of $\zG^0(2)$.  On the other hand,
we will soon see from Figure~\ref{fig:sigmaf} that the images of
$\zr_1$, $\zr_2$ and $\zr_3$ in $\text{PGL}(2,\mathbb{Z})$ generate a
subgroup with index 6.  (Its intersection with
$\text{PSL}(2,\mathbb{Z})$ equals the image of $\zG(2)$.)  It easily
follows that $G_{f_0}$ is generated by $\zr_1$, $\zr_2$, $\zr_3$ together
with the elements of $G_{f_0}$ which lift to translations in
$\text{Aff}(2,\mathbb{Z})$.  One finally verifies that only the
identity element of $G_{f_0}$ lifts to a translation.  Therefore $\zr_1$,
$\zr_2$ and $\zr_3$ generate $G_{f_0}$.
\medskip

\noindent\textbf{Evaluation of $\zm_f$.} Let $\zm_{f_0}$ denote the usual
slope function which $f_0$ induces on slopes of simple closed curves in
$S^2-P(f_0)$.  We want to evaluate $\zm_{f_0}$ at $-1$, $-\frac{1}{2}$, 1
and $\infty$.  This can be done using Theorem 5.1 of \cite{cfpp},
which provides a method suitable for computer implementation.  In the
next paragraph we evaluate $\zm_{f_0}(-\frac{1}{2})$ in a more topological
way.  The other three evaluations can be made similarly.

The left side of Figure~\ref{fig:minushalf} shows the pullback to the
diagram in Figure~\ref{fig:sigmapren} of a simple closed curve $\zg$
in $S^2-P(f_0)$ with slope $-\frac{1}{2}$; it is straightforward to check this
by taking the image of the curves under $f_0$.  The right side of
Figure~\ref{fig:minushalf}, while actually meaningless, might be
helpful.  We see that $f_0^{-1}(\zg)$ has two connected components,
one drawn with dots and the other drawn with dashes.  The dotted
connected component is peripheral.  On the other hand, the image in
$S^2$ of the line segment joining $(0,0)$ and $(1,0)$ is a core arc
for the other connected component.  Recall that $f=h\circ g$ where $g$
is Euclidean and $h$ is a push map.  The slope of this core arc
relative to $P(f_0)$ equals the slope of its image under $h^{-1}$
relative to $P(g)$.  This image under $h^{-1}$ is homotopic rel $P(g)$
to the line segment joining $(0,0)$ and $(2,1)=1\cdot (2,0)+1\cdot
(0,1)$.  Hence $\zm_{f_0}(-\frac{1}{2})=\frac{1}{1}=1$.  In the same
way, we find that $\zm_{f_0}(-1)=0$, $\zm_f(1)=2$ and
$\zm_{f_0}(\infty)=\infty$.
\medskip

\begin{figure}
\begin{center}\includegraphics{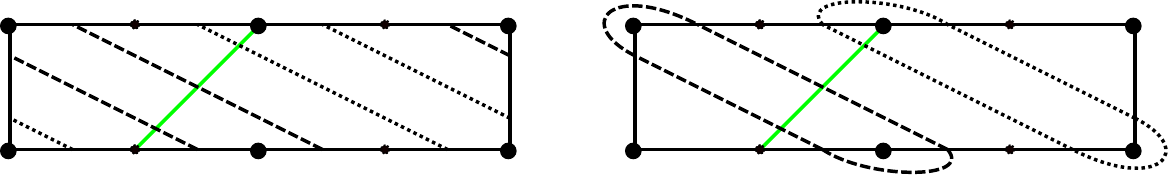}\end{center}
\caption{The pullback of a simple closed curve with slope $-\frac{1}{2}$}
\label{fig:minushalf}
\end{figure}

\noindent\textbf{Evaluation of $\zm_\zv$.} We continue with some
generalities on affine maps. For an extended modular group element
$\zv\in G$ let $\zm_\zv$ denote the induced map on slopes.  Suppose
that $\zv$ lifts to $\zF\in \text{Aff}(2,\mathbb{Z})$ with linear part
given by the matrix $A=\left[\begin{smallmatrix}a & b \\ c & d
\end{smallmatrix}\right]$.  Then $A^{-1}$ maps the line through
$(0,0)$ and $(q,p)$ to the line through $(0,0)$ and
  \begin{equation*}
\left[\begin{smallmatrix}d & -b \\ -c & a \end{smallmatrix}\right]\cdot
(q,p)=(-bp+dq,ap-cq).
  \end{equation*}
So $\zm_\zv(s)=\frac{as-c}{-bs+d}$, where $s=\frac{p}{q}$.
\medskip

\noindent\textbf{Evaluation of $\zs_\zv$.} Still focusing on
generalities, let $\zs_\zv$ denote the map on $\mathbb{H}$ induced by
an extended modular group element $\zv\in G$.  To determine $\zs_\zv$,
it suffices to determine the action of $\zs_\zv$ on $\partial
\mathbb{H}$.  We use the expression for $\zm_\zv$ in the previous
paragraph.  Because the slope $s$ corresponds to $x=-\frac{1}{s}\in
\partial \mathbb{H}$, the map which $\zv$ induces on $\partial
\mathbb{H}$ is
  \begin{equation*}
\zs_\zv(x)=-\left(\frac{a(-1/x)-c}{-b(-1/x)+d}\right)^{-1}=
\frac{dx+b}{cx+a}.
  \end{equation*}
\medskip
So if $z\in \mathbb{H}$, then
  \begin{equation*}
\zs_\zv(z)=\frac{dz+b}{cz+a}\quad\text{ if }\zv\in G^+ \text{ and }
\quad\zs_\zv(z)=\frac{d \overline{z}+b}{c \overline{z}+a} \quad\text{ if
}\zv\notin G^+.
  \end{equation*}

\noindent\textbf{The map $f_0$ induces a virtual endomorphism
$\phi_{f_0}: G \dashrightarrow G$.} For a general NET map $f$, given
an extended modular group element $g \in G$, a lift of $g$ under $f$
might not be unique.  Thus lifting under $f$ maps liftable elements to
cosets of $\text{DeckMod}(f)$, the subgroup of $G$ represented by deck
transformations of $f$.  Proposition 2.4 of \cite{fpp2} implies that
$\text{DeckMod}(f)$ is isomorphic to the group of translations in
$2\zL_2$ which stabilize the Hurwitz structure set $\mathcal{HS}$
modulo the group of translations in $2\zL_1$.  For our example $f_0$,
this quotient group is trivial; we obtain a well-defined virtual
endomorphism $\phi_{f_0}: G \dashrightarrow G$ induced by $f_0$.
\medskip

\noindent\textbf{The extended modular group virtual endomorphism maps
reflections to reflections.} A NET map $f$ preserves
orientation. Therefore, if $g \in G$ reverses orientation and
$\tilde{g}$ is any lift of $g$ under $f$, then $\tilde{g}$ must
reverse orientation.

For our example map $f_0$, since $\phi_{f_0}$ is a homomorphism, if $g
\in G$ is a reflection, then $g$ has order two, therefore
$\phi_{f_0}(g)$ both has order $2$ and reverses orientation, and is
therefore again a reflection.  Thus in terms of the action of liftable
extended mapping class elements $G_{f_0}$ on $\mathbb{H}$, reflections
map to reflections under $\phi_{f_0}$.
\medskip

\noindent\textbf{Evaluation of the extended modular group virtual
endomorphism.}  We continue to focus on our example map $f_0$.  If
$\zv\in G_{f_0}$, then we let $\widetilde{\zv}=\phi_{f_0}(\zv)$.  In
this paragraph we evaluate $\zm_{\zr_i}$, $\zs_{\zr_i}$,
$\zm_{\widetilde{\zr}_i}$ and $\zs_{\widetilde{\zr}_i}$ for $i\in
\{1,2,3\}$.  Using the formulas for $\zm_{\zr_1}$ and $\zs_{\zr_1}$
above, we obtain the leftmost two equations in line~\ref{lin:rhoone}.
We next apply the identity $\zm_f\circ
\zm_\zv=\zm_{\widetilde{\zv}}\circ \zm_{f_0}$ for $\zv\in G_{f_0}$.
Combining this with $\zm_{\zr_1}$ and our values for $\zm_{f_0}$, we
obtain the commutative diagram in line~\ref{lin:rhoone}.  Using the
bottom map of the commutative diagram and the fact that the extended
modular group virtual endomorphism maps reflections to reflections, we
easily obtain the rightmost two equations in line~\ref{lin:rhoone}.
We verify the information in lines~\ref{lin:rhotwo} and
\ref{lin:rhothree} similarly.
\begin{equation}\label{lin:rhoone}
\begin{aligned}
\zm_{\zr_1}(s)= & -s\\
\zs_{\zr_1}(z)= & -z
\end{aligned}
\qquad
\begin{CD}
\infty,-1 @>\zm_{\zr_1}>> \infty,1\\
@V\zm_{f_0}VV @VV\zm_{f_0}V\\
\infty,0 @>\zm_{\widetilde{\zr}_1}>> \infty,2
\end{CD}
\qquad
\begin{aligned}
\zm_{\widetilde{\zr}_1}(s)= & -s+2\\
\zs_{\widetilde{\zr}_1}(z)= & \frac{-z}{2z+1}
\end{aligned}
  \end{equation}

 \begin{equation}\label{lin:rhotwo}
\begin{aligned}
\zm_{\zr_2}(s)= & -s-2\\
\zs_{\zr_2}(z)= & \frac{z}{2z-1}
\end{aligned}
\qquad
\begin{CD}
\infty,-1 @>\zm_{\zr_2}>> \infty,-1\\
@V\zm_{f_0}VV @VV\zm_{f_0}V\\
\infty,0 @>\zm_{\widetilde{\zr}_2}>> \infty,0
\end{CD}
\qquad
\begin{aligned}
\zm_{\widetilde{\zr}_2}(s)= & -s\\
\zs_{\widetilde{\zr}_2}(z)= & -z
\end{aligned}
  \end{equation}

 \begin{equation}\label{lin:rhothree}
\begin{aligned}
\zm_{\zr_3}(s)= & \frac{-s}{2s+1}\\
\zs_{\zr_3}(z)= & -z+2
\end{aligned}
\qquad
\begin{CD}
\infty,-1 @>\zm_{\zr_3}>> -\frac{1}{2},-1\\
@V\zm_{f_0}VV @VV\zm_{f_0}V\\
\infty,0 @>\zm_{\widetilde{\zr}_3}>> 1,0
\end{CD}
\qquad
\begin{aligned}
\zm_{\widetilde{\zr}_3}(s)= & \frac{s}{s-1}\\
\zs_{\widetilde{\zr}_3}(z)= & -z-1
\end{aligned}
  \end{equation}
\medskip

\noindent\textbf{Fundamental domains for $G_{f_0}$, $G_{f_0}^+$,
$\widetilde{G}_{f_0}$ and $\widetilde{G}_{f_0}^+$.} Still focusing on
our example map $f_0$, we now have explicit expressions for
$\zs_{\zr_i}$ and $\zs_{\widetilde{\zr}_i}$ for $i\in \{1,2,3\}$.
These maps are all reflections.  Let $\za_i$ and $\widetilde{\za}_i$
denote the reflection axes of $\zs_{\zr_i}$ and
$\zs_{\widetilde{\zr}_i}$ (their fixed point sets) for $i\in
\{1,2,3\}$.  One verifies that the unshaded triangle in the left side
of Figure~\ref{fig:sigmaf} is a fundamental domain for the action of
$G_{f_0}$ on $\mathbb{H}$.  The two triangles in the left side of
Figure~\ref{fig:sigmaf} form a fundamental domain for the action of
$G_{f_0}^+$ on $\mathbb{H}$.  This shows that the image of $G_{f_0}$ in
$\text{PGL}(2,\mathbb{Z})$ and the image of $G_{f_0}^+$ in
$\text{PSL}(2,\mathbb{Z})$ both have index 6.  The right side of
Figure~\ref{fig:sigmaf} shows fundamental domains for
$\widetilde{G}_{f_0}$ and $\widetilde{G}_{f_0}^+$, the images of $G_{f_0}$ and
$G_{f_0}^+$ under the extended modular group virtual endomorphism.
\medskip

  \begin{figure}
\centerline{\includegraphics{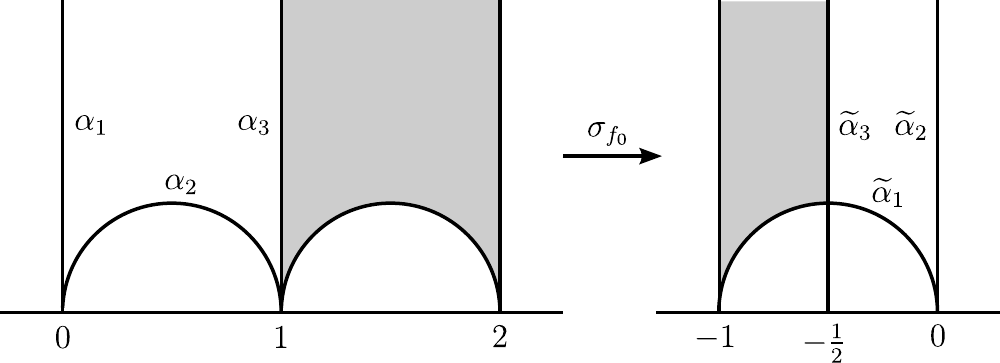}}
\caption{The pullback map $\zs_{f_0}$}
\label{fig:sigmaf}
  \end{figure}

\noindent\textbf{Images of cusps under $\zs_{f_0}$.} For our example
map $f_0$, we have that $\zs_{f_0}\circ
\zs_{\zr_i}=\zs_{\widetilde{\zr}_i}\circ \zs_{f_0}$ for $i\in
\{1,2,3\}$.  It follows that $\zs_{f_0}(\za_i)\subseteq
\widetilde{\za}_i$ for $i\in \{1,2,3\}$.  Using the continuity of
$\zs_{f_0}$ on the Weil-Petersson completion $\mathbb{H}^*$ of
$\mathbb{H}$, it follows that $\zs_{f_0}$ maps the common endpoint 0
of $\za_1$ and $\za_2$ to the common endpoint 0 of $\widetilde{\za}_1$
and $\widetilde{\za}_2$.  (We already essentially knew this through
evaluation of $\zm_{f_0}$.)  Similarly, $\zs_f(1)=\infty$ and
$\zs_{f_0}(\infty)=-\frac{1}{2}+\frac{1}{2}i$.  This last equation
shows that in special cases such as this, it is possible to determine
the image in $\mathbb{H}$ of a cusp under the pullback map of a NET
map.
\medskip

\noindent\textbf{The degree of the induced map $\hat{\zs}_f\co
\mathbb{H}^*/G_f^+\to \mathbb{H}^*/\widetilde{G}_f^+$.} Suppose $f$ is
a general NET map. We make the further assumption that $\zs_f$ is nonconstant.
In this case, the map
$\zs_f\co \mathbb{H}^*\to \mathbb{H}^*$ induces a nonconstant map
$\hat{\zs}_f\co \mathbb{H}^*/G_f^+\to \mathbb{H}^*/\widetilde{G}_f^+$
of compact Riemann surfaces.  The degree of $\hat{\zs}_f$ can be
calculated as follows.  Let $y$ be a cusp in
$\mathbb{H}^*/\widetilde{G}_f^+$.  Then
  \begin{equation*}
\deg(\hat{\zs}_f)=\sum_{\hat{\zs}_f(x)=y}^{}\deg_x(\hat{\zs}_f).
  \end{equation*}
  
Returning to our example map $f_0$, we take $y$ to be the image in
$\mathbb{H}^*/\widetilde{G}_{f_0}^+$ of 0.  Then there is only one
value for $x$, the image of 0 in $\mathbb{H}^*/G_{f_0}^+$.  Because
the multiplier for slope $\infty$ ($=-\frac{1}{0}$) is 1, we see that
a generator of the stabilizer of 0 in $G_{f_0}^+$ maps to a generator
of the stabilizer of 0 in $\widetilde{G}_{f_0}^+$.  Thus
$\deg(\hat{\zs}_{f_0})=1$.
\medskip

\noindent\textbf{Construction of $\zs_{f_0}$.}  Finally, we determine 
$\sigma_{f_0}$.  We first construct a candidate $\zs\co \mathbb{H}\to
\mathbb{H}$, which we will eventually see is $\zs_{f_0}$.  By the Riemann
mapping theorem there exists a unique analytic bijection $\zs$ which
maps the hyperbolic triangle with vertices 0, 1 and $\infty$ to the
hyperbolic triangle with vertices 0, $\infty$ and
$-\frac{1}{2}+\frac{1}{2}i$ so that $\zs(0)=0$, $\zs(1)=\infty$ and
$\zs(\infty)=-\frac{1}{2}+\frac{1}{2}i$.  We then extend the
definition of $\zs$ to all of $\mathbb{H}$ using the reflection
principle.  This defines $\zs\co \mathbb{H}\to \mathbb{H}$.

Just as $\zs_{f_0}$ induces the map $\hat{\zs}_{f_0}\co
\mathbb{H}^*/G_{f_0}^+\to \mathbb{H}^*/\widetilde{G}_{f_0}^+$, the map
$\zs$ induces a map $\hat{\zs}\co \mathbb{H}^*/G_{f_0}^+\to
\mathbb{H}^*/\widetilde{G}_{f_0}^+$.  These Riemann surfaces both have
genus 0.  Both $\hat{\zs}_{f_0}$ and $\hat{\zs}$ have degree 1 and
they agree at the three cusps.  Thus they are equal.  Therefore the
restriction of $\zs_{f_0}$ to $\za_1$, for example, agrees with the
restriction of $\zs$ to $\za_1$.  It follows that $\zs_{f_0}=\zs$.

\section{Dynamics on curves in degree 2}\label{sec:fga}

In this section we investigate the dynamics on curves for the 
set of NET maps with degree 2 and exactly one critical postcritical
point.  We will then indicate how to extend this
result to all NET maps with degree 2 and hyperbolic orbifolds.

\begin{theorem}\label{thm:fga} Let $g$ be a NET map with degree 2 and
exactly one critical postcritical point.  Let $\zm_g$ be the slope
function of $g$.  Then we have the following.
\begin{enumerate}
  \item If $g$ is combinatorially equivalent to a rational map, then $\zm_g$
has a global attractor containing at most four slopes.
  \item Suppose that $g$ is not combinatorially equivalent to a
rational map.  Let $s\in \overline{\mathbb{Q}}=\mathbb{Q}\cup
\{\infty\}$ be the slope of the obstruction of $g$.  Let $\zh$ be a
generator of the cyclic group of $g$-liftable elements in the modular
group of $g$ which stabilize $s$.  Then under iteration, a slope
either becomes undefined (that is, the corresponding curve is trivial
or peripheral), or lands in either 
\begin{enumerate}
  \item $\{s\}$, or 
  \item $\{s\}\cup \{\zh^m(r):m\in \mathbb{Z}\}$ for some $r\in
\overline{\mathbb{Q}}$ such that $\zm_g(\zh^m(r))=\zh^{m+n}(r)$
for every integer $m$ and some integer $n$.
\end{enumerate}
\end{enumerate}
All three cases occur as well as all possible values of $n$ in case
2b.
\end{theorem}
  \begin{proof} Let $f$ be the NET map $f_0$ of \S \ref{sec:sigmaf}.
As in \S \ref{sec:sigmaf}, let $G_f$ denote the subgroup of
liftables for $f$ in the extended modular group
$\text{EMod}(S^2,P(f))$.

Figure~\ref{fig:sigmaf} in effect describes the pullback map $\zs_f$.
On the left are two triangles in a tesselation $T_f$ of the
Weil-Petersson completion $\mathbb{H}^*$ of $\mathbb{H}$ by
fundamental domains for the action of $G_f$.  The pullback map $\zs_f$
maps the unshaded, respectively shaded, triangle on the left to the
unshaded, respectively shaded, triangle on the right.  Extending by
the reflection principle, we see that $\zs_f$ maps every triangle of
$T_f$ into a triangle of $T_f$.

Theorem~\ref{thm:hclasses} implies that $g$ lies in the same impure 
Hurwitz class as $f$.  If $g$ is conjugate to $f$, then clearly
there is a tesselation $T_g$ of $\mathbb{H}^*$ by fundamental domains
for the action of $G_g$ on $\mathbb{H}^*$ such that $\zs_g$ maps every
triangle of $T_g$ into a triangle of $T_g$.  Suppose that $g=f\circ
\zv$ for some map $\zv$ representing an element of
$\text{Mod}(S^2,P(f))$.  Then $\zs_g=\zs_\zv\circ \zs_f$.  One easily
verifies that $G_g=G_f$ and that $\zs_\zv$ acts as an automorphism of
$T_f$.  Thus $\zs_g$ maps every triangle of $T_g$ into a triangle of
$T_g$ in this case also.  We conclude that the map $g$ of
Theorem~\ref{thm:fga} maps every triangle of $T_g$ into a triangle of
$T_g$.  It follows that every iterate of $\zs_g$ maps every triangle
of $T_g$ into a triangle of $T_g$.

Now suppose that $g$ is combinatorially equivalent to a rational map.
Then $\zs_g$ has a fixed point $\zt\in \mathbb{H}$.  Let $r\in
\overline{\mathbb{Q}}$.  Let $t$ be a triangle of $T_g$ which has $r$
as a vertex.  Let $z$ be a point in the interior of $t$.  Then the
points $z$, $\zs_g(z)$, $\zs_g^2(z),\ldots$ converge to $\zt$.  So
they eventually enter the star of $\zt$ in $T_g$ (the union of
triangles containing $\tau$).  Because iterates of $\zs_g$ map $t$
into triangles of $T_g$, it follows that $t$ eventually enters the
star of $\zt$ in $T_g$.  Because $\zs_g$ is continuous on
$\mathbb{H}^*$, it follows that $r$ eventually enters the star of
$\zt$ in $T_g$.  This star has at most two triangles and at most four
vertices.  Since $\zs_g$ and $\zm_g$ are conjugate on
$\overline{\mathbb{Q}}$ via $p/q \mapsto -q/p$, this proves statement
1.

Now suppose that $g$ is not combinatorially equivalent to a rational
map.  So $g$ has an obstruction.  The pullback map $\zs_g$ fixes the
negative reciprocal of the slope of this obstruction.  We find it
convenient for this fixed point to be $\infty $.  So we replace the
map $f$ two paragraphs above by a conjugate so that the new pullback
map is $\zs_f$ conjugated by $z\mapsto -1/z$.  Arguing as two
paragraphs above, we find that Figure~\ref{fig:sigmag} describes
$\zs_g$ in the same way that Figure~\ref{fig:sigmaf} describes
$\zs_f$.  The $n$ in Figure~\ref{fig:sigmag} is an integer.  Any
integer is possible.  The case $n=0$ is the case in which $g$ is
conjugate to $f$.

  \begin{figure}
\centerline{\includegraphics{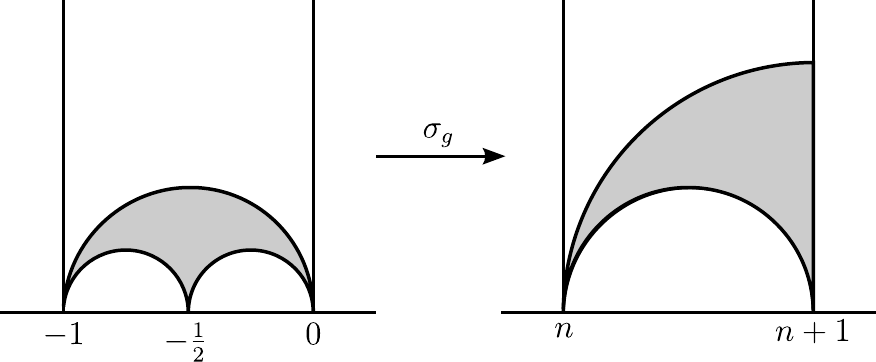}}
\caption{The pullback map $\zs_g$}
\label{fig:sigmag}
  \end{figure}

Arguing as in the case in which $g$ is unobstructed, we find that
every element of $\overline{\mathbb{Q}}$ eventually enters the star in
$T_g$ of $\infty$ under the iterates of $\zs_g$.  Hence it only
remains to determine the action of $\zs_g$ on integers.  

We have that $\zs_g(-1)=n$.  Using the reflection principle, we see
that $\zs_g(m)=m+n+1$ for every odd integer $m$.  Similarly,
$\zs_g(m)=m+n+1+i$ for every even integer $m$.  Furthermore, the
stabilizer of $\infty$ in the subgroup of modular group liftables for
$g$ has a generator which acts on $\mathbb{H}$ as $z\mapsto z+2$.

Now suppose that $n$ is even.  Then $\zs_g$ maps odd integers to even
integers, and it maps even integers into $\mathbb{H}$.  So every
integer eventually leaves $\overline{\mathbb{Q}}$.  We are in the
situation of case 2a.

Finally suppose that $n$ is odd.  Then $\zs_g$ maps odd integers to
odd integers, and it maps even integers into $\mathbb{H}$.  It follows
that we are in case 2$b$ with $n+1$ here being 2 times $n$ there.

The only thing left to prove is that case 1 actually occurs, namely,
that there exists a rational NET map with degree 2 and exactly one
critical postcritical point.  An example of such a map is
$f(z)=z^2+i$. 

This proves Theorem~\ref{thm:fga}.
\end{proof}

It was noted in the above proof that the map $f_0$ of \S
\ref{sec:sigmaf} corresponds to the case $n=0$.  Since 0 is even, the
map $f_0$ falls into case 2a.  Thus $f_0$ provides an example of an
obstructed Thurston map with hyperbolic orbifold whose pullback map on
curves has a finite global attractor consisting of just the
obstruction.

We next indicate how Theorem~\ref{thm:fga} can be extended to all NET
maps with degree 2 and hyperbolic orbifold.
Theorem~\ref{thm:hclasses} and the paragraph preceding it imply that
there are two  impure Hurwitz  classes of NET maps with
hyperbolic orbifolds.  The impure Hurwitz class of maps with one critical
postcritical point is represented by the map $f_0$ of
\S \ref{sec:sigmaf}, and the proof of Theorem~\ref{thm:fga} uses
$f_0$.  The impure Hurwitz class of maps with two critical postcritical points
is represented by the rabbit, and in the same way it is possible to
use the rabbit to prove the corresponding result for these maps.

We discuss this extension in this paragraph.  Let $f(z)=z^2+c_R$
denote the rabbit polynomial of the introduction.
Figure~\ref{fig:rabbitpren} gives a NET map presentation diagram for
$f$.  Arguing as in \S \ref{sec:sigmaf}, we find that
Figure~\ref{fig:sigmarabbit} provides an analog to
Figure~\ref{fig:sigmaf} for the pullback map $\zs_f$ of $f$.  More
precisely, $\zs_f$ maps the unshaded, respectively shaded, triangle in
the left side of Figure~\ref{fig:sigmarabbit} bijectively to the
unshaded, respectively shaded, triangle in the right side of
Figure~\ref{fig:sigmarabbit} with $\zs_f(0)=\infty$,
$\zs_f(\infty)=-1$, $\zs_f(-1)=0$ and
$\zs_f(-2)=-\frac{1}{2}+\frac{1}{2}i$.  The last equation provides
another example of an exact evaluation of a pullback map at an element
of $\overline{\mathbb{Q}}$ when that value lies in $\mathbb{H}$.
Whereas before we worked with a tesselation by triangles, now we work
with a tesselation by quadrilaterals, those determined by the union of
the two triangles in the left side of Figure~\ref{fig:sigmarabbit}.
From here the argument proceeds as before.  The result is essentially
the same, although the statement must be modified a bit.  The main
difference is that case 2a does not occur here.

  \begin{figure}
\centerline{\includegraphics{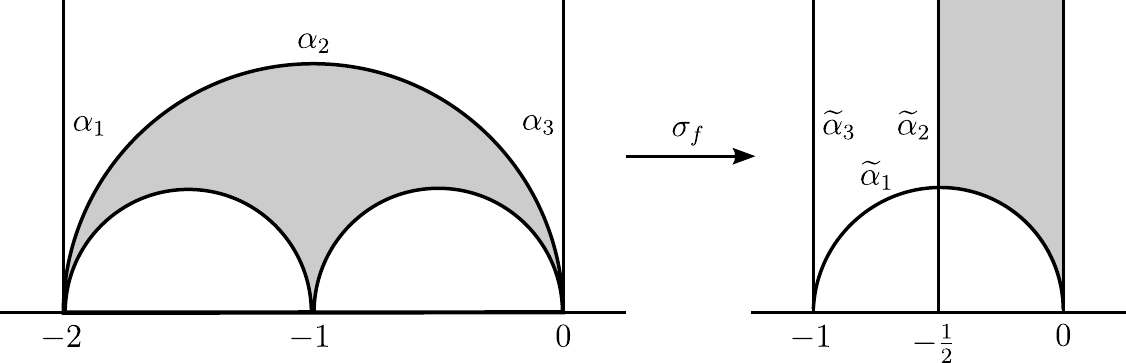}}
\caption{The rabbit's pullback map $\zs_f$}
\label{fig:sigmarabbit}
  \end{figure}

We say a few more words about the pullback map $\zs_f$ for the rabbit
in this paragraph.  The blue curve in the left side of
Figure~\ref{fig:julias} has slope 0 relative to
Figure~\ref{fig:rabbitpren}.  This curve lies in a 3-cycle of curves
for the pullback map on curves.  In terms of slopes, $\zm_f(0)=1$,
$\zm_f(1)=\infty$ and $\zm_f(\infty)=0$.  Because points of $\partial
\mathbb{H}$ correspond to negative reciprocals of slopes, these
equations correspond to the equations $\zs_f(\infty)=-1$,
$\zs_f(-1)=0$ and $\zs_f(0)=\infty$, just as in the previous
paragraph.  Because $\zs_f$ maps the quadrilateral with vertices $-2$,
$-1$, 0 and $\infty$ into itself, its fixed point is in this
quadrilateral.  Figure~\ref{fig:rabbithalfsp} further shows that this
fixed point has small imaginary part (within the Euclidean circle
centered at $(-1,0)$ with radius $\sqrt{2}$).

Since the fixed point $\zt$ of $\zs_f$ is in the interior of the
quadrilateral which is the union of the two triangles on the left side
of Figure~\ref{fig:sigmarabbit}, the star of $\zt$ in this tesselation
consists of just this quadrilateral. So for any $t\in \overline{\mathbb{Q}}$,
either there exists a positive integer $n$ with $\zm_f^{\circ n}(t) = \odot$,
or for $n$ sufficiently large $\zs_f^{\circ n}(t)\in \{0,-1,\infty\}$. Thus
there is a finite global attractor consisting of points whose slopes
correspond to $0$, $1$ and $\infty$.  This gives another proof of
Theorem~\ref{thm:rabbit}.

\section{The extended half-space theorem}\label{sec:extdhalfsp}

The goal of this section is to sketch a proof of the extended
half-space theorem.  After filling in the details, we obtain an
explicit interval about every such point which contains no negative
reciprocals of obstructions other than possibly that point.  As far as
we know, for every NET map $f$ there exist finitely many intervals
provided by the half-space theorem and finitely many intervals provided
by the extended half-space theorem whose union is a cofinite subset of
$\partial \mathbb{H}$.  The finitely many omitted points are extended
rational numbers.  Here is a qualitative statement of the theorem.

\begin{theorem}[Extended Half-Space Theorem]\label{thm:exhalsfp} Let
$f$ be a NET map with slope function $\zm_f$.  Let $\frac{p}{q}\in
\overline{\mathbb{Q}}$, and suppose that either
$\zm_f(\frac{p}{q})=\frac{p}{q}$ and $\zd_f(\frac{p}{q})\ne 1$ or 
that $\zm_f(\frac{p}{q})=\odot$.  Then
there exists an interval in $\mathbb{R}\cup \{\infty\}$ containing
$-\frac{q}{p}$ which contains no negative reciprocals of obstructions
for $f$ other than possibly $-\frac{q}{p}$.
\end{theorem}

To begin a sketch of the proof of this, we recall the setting of the
half-space theorem.  Let $s_1$ be the slope of a simple closed curve
in $S^2-P(f)$ whose preimage under $f$ contains a connected component
which is essential and nonperipheral (if no such $s_1$ exists, $f$ is
unobstructed, by Thurston's characterization theorem).
Let $s'_1=\zm_f(s_1)$, and
suppose that $s'_1\ne s_1$.  In this situation the half-space theorem
supplies an open half-space $H_1$ in the upper half-plane which
contains no fixed point of $\zs_f$ and whose boundary's interior
contains $-1/s_1$ but no negative reciprocal of a obstruction for $f$.
The half-space $H_1$ depends only on $s_1$, $s'_1$ and the multiplier
$\zd_f(s_1)$.  We call such a half-space an excluded half-space.

Let $t$ be an extended rational number which is not mapped to a
different extended rational number by $\zs_f$.  We will use functional
equations satisfied by $\zs_f$ to produce excluded half-spaces near
$t$ so that the collection of all extended real numbers excluded
by these half-spaces together with $t$ forms an open interval about
$t$ in $\mathbb{R}\cup \{\infty\}$.

We consider the simplest case, the case in which $t=\infty$.  The
general case can be gotten from this by applying an element of
$\text{PSL}(2,\mathbb{Z})$ to $t$.  Keep in mind that points of
$\overline{\mathbb{Q}}$ in the boundary of $\mathbb{H}$ are to be
viewed as negative reciprocals of slopes.  So either $\zm_f(0)=0$ or
$\zm_f(0)=\odot$.  Let $\zg$ be a simple closed curve in $S^2-P(f)$
with slope 0.  Let $d$ be the degree with which $f$ maps every
connected component of $f^{-1}(\zg)$ to $\zg$.  Let $c$ be the number
of these connected components which are neither inessential nor
peripheral. Theorem 7.1 of \cite{cfpp}, for example, yields the
functional equation $\zs_f\circ \zv^d=\zv^c\circ \zs_f$, where
$\zv(z)=z+2$.  We have that $\zd_f(0)=\frac{c}{d}$.  By hypothesis,
$\frac{c}{d}\ne 1$.  For convenience we consider
the case that $\frac{c}{d}<1$, so $c<d$.

Let $t_1=-\frac{1}{s_1}$ and $t'_1=-\frac{1}{s'_1}$, where $s_1$ and
$s'_1$ are as above.  Since $t=\infty$, it is
natural to assume that $t_1\ne \infty$.  For simplicity, we also assume
that $t'_1\ne \infty$.  Let $B_1$ and $B'_1$ be closed horoballs at
$t_1$ and $t'_1$ as in the statement of the half-space theorem
in \cite[Theorem 5.6]{cfpp}.  Let
$r$ be the Euclidean radius of $B_1$, and let $r'$ be the
Euclidean radius of $B'_1$.

If $r>r'$, then our excluded half-space $H_1$ is unbounded in the
Euclidean metric, and so we already have an open neighborhood of
$\infty$ in $\mathbb{R}\cup \{\infty\}$ which contains no negative
reciprocals of obstructions.  So we assume that $r\le r'$.  The case
in which $r=r'$ can be handled as follows.  In this case $H_1$ is
bounded by a vertical Euclidean ray with endpoint the average value of
$t_1$ and $t'_1$.  This gives us an unbounded interval of real numbers
which contains no negative reciprocals of obstructions.  Using the
fact that $c<d$, we replace $t_1$ and $t'_1$ by their images under an
appropriate power of $\zv^d$ and $\zv^c$ (possibly negative) so that
the order of these images is opposite to the order of $t_1$ and
$t'_1$.  The resulting excluded half-space and $H_1$ combine to
produce an open neighborhood of $\infty$ in $\mathbb{R}\cup
\{\infty\}$ containing no negative reciprocals of obstructions.  This
establishes the existence of such an interval.  Hence we assume that
$r<r'$.  In this case $H_1$ lies within a Euclidean semicircle.  Let
$C_1$ and $R_1$ be the center and radius of this semicircle.  See
Figure~\ref{fig:horoballs}, which assumes that $t_1>t'_1$.

  \begin{figure}
\centerline{\includegraphics{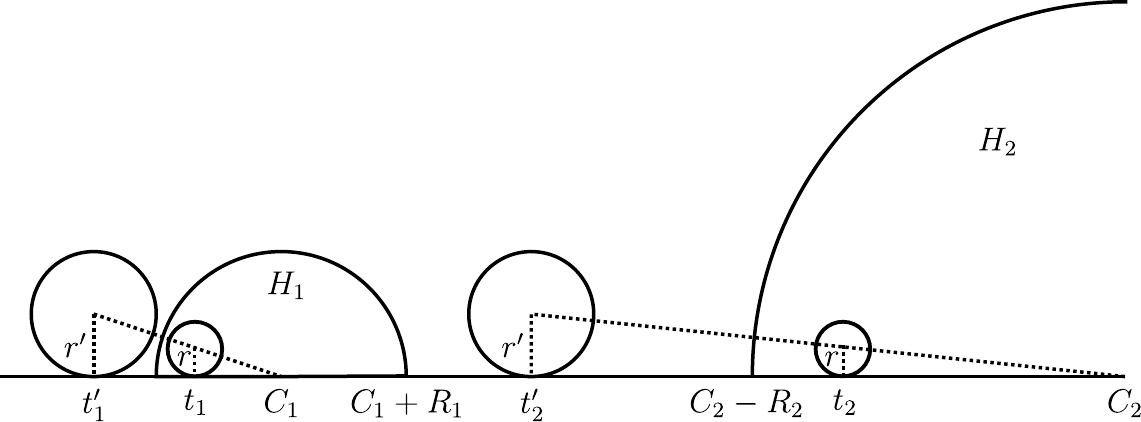}}
\caption{The basic diagram for the extended half-space theorem}
\label{fig:horoballs}
  \end{figure}

Now we apply the functional equation $\zs_f\circ
\zv^d=\zv^c\circ \zs_f$.  Set $t_2=\zv^d(t_1)$ and $t'_2=\zv^c(t'_1)$.
Then the equation $\zm_f(s_1)=s'_1$ implies that $\zs_f(t_1)=t'_1$,
and so $\zs_f(t_2)=t'_2$.  Because $t_1>t'_1$, and $d>c$, we have that
$t_2>t'_2$.

In this paragraph we show that the half-space theorem applies to $t_2$
and $t'_2$ using the horoballs $B_2=\zv^d(B_1)$ and
$B'_2=\zv^c(B'_1)$.  We have the equation $\zs_f\circ \zv^d=\zv^c\circ
\zs_f$.  This is induced by a homotopy equivalence of the form
$\zf^d\circ f\sim f\circ \zf^c$, where $\zf$ is a Dehn twist about
a curve with slope $0$.  The
equation $t_2=\zv^d(t_1)$ implies that if $\zg_1$ is a simple closed
curve in $S^2-P (f)$ with slope $s_1=-1/t_1$, then
$\zg_2=\zf^{-d}(\zg_1)$ is a simple closed curve in $S^2-P(f)$ with
slope $s_2=-1/t_2$.  A corresponding statement holds for $\zf^{-c}$.
Now the homotopy equivalence $\zf^d\circ f\sim f\circ \zf^c$ shows
that if $c_2$ is the number of connected components of $f^{-1}(\zg_2)$
which are neither inessential nor peripheral and if $d_2$ is the
degree with which $f$ maps these components to $\zg_2$, then $c_2=c_1$
and $d_2=d_1$.  So $\zd_f(s_2)=\zd_f(s_1)$.  Combining this with
Corollary 6.2 of \cite{cfpp}, which shows how elements of
$\text{PGL}(2,\mathbb{Z})$ map horoballs to horoballs, it follows that
the half-space theorem applies to $t_2$ and $t'_2$ using the horoballs
$B_2=\zv^d(B_1)$ and $B'_2=\zv^c(B'_1)$.  Hence we obtain another
excluded half-space $H_2$ corresponding to the horoballs $B_2$ and
$B'_2$ at $t_2$ and $t'_2$ with Euclidean radii $r$ and $r '$.

We want $H_1\cap H_2\ne \emptyset$ because then the open intervals in
$\mathbb{R}$ determined by $H_1$ and $H_2$ can be combined to form a
larger interval.  Since $t_2>t_1$ as in Figure~\ref{fig:horoballs},
$H_1\cap H_2\ne \emptyset$ if and only if $C_1+R_1>C_2-R_2$.  We make
an explicit computation based on this and find that
  \begin{equation*}
H_1\cap H_2\ne \emptyset\Longleftrightarrow
t_1-t'_1>d\left(\sqrt{r'/r}-1\right)+c\left(1-\sqrt{r/r'}\right).
  \end{equation*}

Suppose that the last inequality is satisfied.  Then because
  \begin{equation*}
t_2-t'_2=t_1+2d-t'_1-2c=t_1-t'_1+2(d-c)>t_1-t'_1,
  \end{equation*}
the inequality in the next-to-last display is satisfied with
$t_1-t'_1$ replaced by $t_2-t'_2$.  Inductively, we conclude that if
$H_1\cap H_2\ne \emptyset$, then $f$ has no obstruction $s$ with
$-\frac{1}{s}>C_1-R_1$.  Furthermore, the last display shows that the
differences $t_1-t'_1$ increase without bound under iteration, and so
it is possible to find $t_1$ such that $H_1\cap H_2\ne \emptyset$.

This obtains an unbounded interval of positive real numbers which
contains no negative reciprocals of obstructions.  Symmetry yields a
corresponding interval of negative real numbers.  This is the gist of
the extended half-space theorem.  It remains to make the estimates
explicit for computation.  This is a bit tedious, but straightforward.

Theorem \ref{thm:exhalsfp} is false if $\delta_f(p/q)=1$; counterexamples are found 
among maps in 21HClass3 and 31Hclass5, 6, 9.

\section{Acknowledgements}

The authors gratefully acknowledge support from the American Institute
for Mathematics, which funded the ``SQuaRE'' workshops in the summers
of 2013--15 during which much of this joint work was conducted.
Russell Lodge was also supported by the Deutsche Forschungsgemeinschaft.
Kevin Pilgrim was also supported by Simons grant \#245269. Sarah Koch was also supported by the NSF and the Sloan Foundation.


\begin{thebibliography}{99}

\bibitem{B}
L.~Bartholdi,
\emph{IMG},
software package for GAP,
\url{https://github.com/laurentbartholdi/img}, 2014.

\bibitem{BBY}
S.~Bonnot, M.~Braverman, and M.~Yampolsky,
\emph{Thurston equivalence to a rational map is decidable},
Mosc. Math. J. {\bf 12} (2012), no. 4, 747"1¤7763, 884. 

\bibitem{BD}
L.~Bartholdi and D.~Dudko,
\emph{Algorithmic aspects of branched coverings},
arxiv: \url{http://arxiv.org/abs/1512.05948}, 2016.

\bibitem{BN}
L.~Bartholdi and V.~Nekrashevych,
\emph{Thurston equivalence of topological polynomials}, 
Acta Math. \textbf{197} (2006), 1--51.

\bibitem{BEKP}
X.~Buff, A.~Epstein, S.~Koch, and K.~Pilgrim,
\emph{On Thurston's pullback map},
In {\em Complex Dynamics--Families and Friends}, pp. 561-583.  
A. K. Peters, Wellesley, MA 2009. 

\bibitem{fsrcve}
J.~W.~Cannon, W.~J.~Floyd, W.~R.~Parry, and K.~M.~Pilgrim 
\emph{Subdivision rules and virtual endomorphisms}, Geom. Dedicata
\textbf{141} (2009), 181--195.

\bibitem{cfpp}
J.~W.~Cannon, W.~J.~Floyd, W.~R.~Parry and K.~M.~Pilgrim, {\em Nearly
Euclidean Thurston maps}, Conform. Geom. Dyn. \textbf{16} (2012),
209--255 (electronic).

\bibitem{kmp:fixed}
K.~Cordwell, S.~Gilbertson, N.~Nuechterlein, K.~M.~Pilgrim, and S.~Pinella
\emph{On the classification of critically fixed rational maps}.
 Conform. Geom. Dyn. \textbf{19} (2015),
51--94 (electronic).

\bibitem{DH}
A.~Douady and J.~H.~Hubbard,
\emph{A proof of Thurston's topological characterization of
rational functions},
Acta Math. {\bf 171} (1993), 263--297.

\bibitem{fpp1}
W.~J.~Floyd, W.~R.~Parry, and K.~M.~Pilgrim, {\em Presentations of NET
maps}, in preparation.

\bibitem{fpp2}
W.~J.~Floyd, W.~R.~Parry, and K.~M.~Pilgrim, {\em Modular groups,
Hurwitz classes and dynamic portraits of NET maps}, in
preparation.

\bibitem{kam}
Atsushi~Kameyama, \emph{The Thurston equivalence for postcritically
finite branched coverings}, Osaka J. Math. {\bf 38} (2001), 565--610.

\bibitem{KL}
G.~Kelsey and R.~Lodge, \emph{Quadratic nearly Euclidean Thurston
maps}, in preparation.

\bibitem{K}
S.~Koch, \emph{Teichm\"{u}ller theory and critically finite
endomorphisms}, Advances in Math. {\bf 248} (2013), 573--617.

\bibitem{kps}
S.~Koch, K.~M.~Pilgrim and N.~Selinger, {\em Pullback
invariants of Thurston maps}, {\bf 368} (2016), 4621--4655.

\bibitem{leh}
J.~Lehner, {\em Discontinuous groups and automorphic functions},
Math. Surveys \textbf{8}, Amer. Math. Soc., Providence, 1964.

\bibitem{lev}
S.~Levy, \emph{Critically finite rational maps}, PhD thesis,
Princeton University, 1985.


\bibitem{L}
R.~Lodge,
\emph{Boundary values of the Thurston pullback map},
Conform. Geom. Dyn. \textbf{19} (2015), 77--118 (electronic).

\bibitem{Mey}
D.~Meyer,
\emph{Unmating of rational maps, sufficient criteria and examples},
In {\em Frontiers in Complex Dynamics: In Celebration of John Milnor's 
80th Birthday}, pp. 197-234.  Princeton University Press, 2014. 

\bibitem{Mi}
J.~Milnor,
\emph{Pasting together Julia sets: a worked-out example of mating},
Experiment. Math. {\bf 13} (2004), no. 1, 55"1¤792. 

\bibitem{N}
V.~Nekrashevych, {\em Self-Similar Groups},
Math. Surveys and Monographs \textbf{117},
Amer. Math. Soc., Providence, 2005.

\bibitem{NET}
The NET map web site, \url{www.math.vt.edu/netmaps/}.

\bibitem{kmp:combinations}
K.~M.~Pilgrim, {\em Combinations of complex dynamical systems}. Springer Lecture Notes in Mathematics {\bf 1827}, 2003.

\bibitem{kmp:tw}
K.~M.~Pilgrim, {\em An algebraic formulation of Thurston's
characterization of rational functions}, Annales de la Facult\'e
des Sciences de Toulouse {\bf XXI} No. 5 (2012), 1033--1068.

\bibitem{SM}
E.~A.~Saenz Maldonado, \emph{On nearly Euclidean Thurston maps},
Ph.D. Thesis, Virginia Tech, 2012.

\bibitem{S}
N.~Selinger, {\em Thurston's pullback map on the augmented
Teichm\"{u}ller space and applications}, Invent. Math. {\bf 189}
(2012), 111--142.

\end{thebibliography}
\end{document}